\documentclass[11pt]{article}
\usepackage{amsthm}
\usepackage{amsfonts}
\usepackage{amssymb}
\usepackage{amsmath}
\usepackage{enumitem}
\allowdisplaybreaks
\usepackage{cite}
\usepackage{epsfig}
\usepackage[top=1in,bottom=1in,left=1in,right=1in]{geometry}

\newcommand\dd{\,\mbox{d}}
\newcommand\AAA{{\cal H}}
\newcommand\NN{{\mathbb N}}
\newcommand\RR{{\mathbb R}}
\newcommand\jj{{\mathfrak j}}
\newcommand\cut[1]{\|#1\|_{\square}}
\newcommand\numv[1]{v\left(#1\right)}
\newcommand\nume[1]{e\left(#1\right)}

\newtheorem{theorem}{Theorem}
\newtheorem{corollary}[theorem]{Corollary}
\newtheorem{lemma}[theorem]{Lemma}
\newtheorem{conj}[theorem]{Conjecture}
\newtheorem{problem}[theorem]{Problem}
\newtheorem{claim}{Claim}[theorem]
\newtheorem*{problem*}{Problem}

\DeclareTextCompositeCommand{\v}{OT1}{l}{l\nobreak\hspace{-.1em}'}
\DeclareTextCompositeCommand{\v}{OT1}{t}{t\nobreak\hspace{-.1em}'\nobreak\hspace{-.15em}}

\begin{document}
\title{Common graphs with arbitrary chromatic number
       \thanks{The work of the first author has received funding from the European Research Council (ERC) under the European Union's Horizon 2020 research and innovation programme (grant agreement No 648509).
               This publication reflects only its authors' view; the European Research Council Executive Agency is not responsible for any use that may be made of the information it contains.
	       The second author is supported by the grant 23-06815M of the Grant Agency of the Czech Republic.
	       The first and second authors were also supported by the MUNI Award in Science and Humanities (MUNI/I/1677/2018) of the Grant Agency of Masaryk University.
	       The third author is supported by NSF Grant DMS-1953958.}}

\author{Daniel Kr{\'a}\v{l}\thanks{Faculty of Informatics, Masaryk University, Botanick\'a 68A, 602 00 Brno, Czech Republic. E-mail: {\tt dkral@fi.muni.cz}.}\and
        Jan Volec\thanks{Department of Mathematics, Faculty of Nuclear Sciences and Physical Engineering, Czech Technical University in Prague, Trojanova 13, 120 00 Prague, Czech Republic. E-mail: {\tt jan@ucw.cz}.}\and
	Fan Wei\thanks{Duke Mathematics, Campus Box 90320, Durham, NC 27708, USA. E-mail: {\tt fan.wei@duke.edu}. Previous affiliation: Department of Mathematics, Princeton University, Princeton, NJ 08543, USA.}}

\date{} 
\maketitle

\begin{abstract}
Ramsey's Theorem guarantees for every graph $H$ that
any $2$-edge-coloring of a sufficiently large complete graph contains a monochromatic copy of $H$.
In 1962, Erd\H os conjectured that
the random $2$-edge-coloring minimizes the number of monochromatic copies of $K_k$, and
the conjecture was extended by Burr and Rosta to all graphs.
In the late 1980s,
the conjectures were disproved by Thomason and Sidorenko, respectively.
A classification of graphs whose number of monochromatic copies is minimized by the random $2$-edge-coloring,
which are referred to as common graphs, remains a challenging open problem.
If Sidorenko's Conjecture, one of the most significant open problems in extremal graph theory, is true,
then every $2$-chromatic graph is common, and in fact,
no $2$-chromatic common graph unsettled for Sidorenko's Conjecture is known.
While examples of $3$-chromatic common graphs were known for a long time,
the existence of a $4$-chromatic common graph was open until 2012, and
no common graph with a larger chromatic number is known.

We construct connected $k$-chromatic common graphs for every $k$.
This answers a question posed by Hatami, Hladk\'y, Kr\'al', Norine and Razborov [Combin. Probab. Comput. 21 (2012), 734--742], and
a problem listed by Conlon, Fox and Sudakov [London Math. Soc. Lecture Note Ser. 424 (2015), 49--118, Problem 2.28].
This also answers in a stronger form the question raised by Jagger, \v S\v tov\'\i\v cek and Thomason [Combinatorica 16, (1996), 123--131]
whether there exists a common graph with chromatic number at least four.
\end{abstract}


\section{Introduction}
\label{sec:intro}

Ramsey's Theorem~\cite{Ram30}
led to a large body of results on the existence of well-behaved substructures in large structures, known as Ramsey Theory,
with links to many other areas of mathematics, see e.g.~\cite{GraRS13, ConFS15}.
Prominent examples of such links are Hindman's Theorem~\cite{Hin74} and Szemer\'edi's Theorem~\cite{Sze75} in number theory.
In one of its simplest forms, the Ramsey's Theorem asserts that for every graph $H$,
there exists an integer $N$ such that
any $2$-edge-coloring of the complete graph with $N$ vertices contains a monochromatic copy of $H$.
Determining the smallest such $N$,
known as the Ramsey number $r(H)$ of a graph $H$, is a famous open problem,
even in the supposedly simplest case when $H$ is a complete graph.
The best lower bound on the Ramsey number of a complete graph is by a random construction,
which pioneered the development of the probabilistic method in combinatorics~\cite{AloS16},
however, whether this random construction is (asymptotically) optimal is still widely open 
despite recent major progress~\cite{Con09,Sah23} and,
in particular, the recent major breakthrough~\cite{CamGMS23} on the upper bound.
This problem is an example of a classical theme in extremal combinatorics:
when is a random graph construction (close to) optimal?
In this paper, we are concerned with the quantitative version of this problem,
known as Ramsey multiplicity,
which asks how many monochromatic copies of a graph $H$
necessarily exist in any $2$-edge-coloring of the complete graph with $N$ vertices, and
when the bound coming from the random construction is optimal.

Goodman's Theorem~\cite{Goo59} implies that the number of monochromatic copies of the triangle $K_3$
is asymptotically minimized by the random $2$-edge-coloring,
i.e., when each edge of a complete graph is colored randomly with one of two colors with probability $1/2$.
We say that a graph $H$ is \emph{common} if the number of monochromatic copies of $H$
is asymptotically minimized by the random $2$-edge-coloring of a complete graph.
In particular, $K_3$ is common and more generally every cycle is common~\cite{Sid89}.

In 1962, Erd\H{o}s~\cite{Erd62} conjectured that every complete graph is common, and
later Burr and Rosta~\cite{BurR80} conjectured that every graph is common.
Both conjectures turned out to be false:
in the late 1980s,
Sidorenko~\cite{Sid86,Sid89} showed that a triangle with a pendant edge is not common and
Thomason~\cite{Tho89} showed that $K_4$ is not common.
More generally, any graph containing $K_4$ is not common~\cite{JagST96} (and so almost every graph is not common), and
there exist graphs $H$ and $2$-edge-colorings with the number of monochromatic copies of $H$
sublinear in the number of monochromatic copies of $H$ in the random $2$-edge-coloring~\cite{Cla92,Fox08}.

A characterization of the class of common graphs is an intriguing open problem and
there is even no conjectured description of the class.
This problem is closely related to the famous conjecture of Sidorenko~\cite{Sid93} and of Erd\H{o}s and Simonovits~\cite{ErdS83},
which asserts that every bipartite graph $H$ has the Sidorenko property,
i.e., the number of copies of $H$ in any graph
is asymptotically at least the number of its copies in the random graph with the same density.
Since every graph $H$ with the Sidorenko property is common,
as the number of copies of $H$ in each color class is at least the expected number of its copies in the random edge-coloring,
the conjecture, if true, would imply that all bipartite are common.
Hence, families of bipartite graphs known to have the Sidorenko property~\cite{BlaR65, Sid89, Sid91, ConFS10, ConL17, ConKLL18, ConL21}
provide examples of bipartite graphs that are common.

Common graphs that are not bipartite, i.e., their chromatic number is larger than two, are scarce.
In particular, Jagger, \v S\v tov\'\i\v cek and Thomason asked
whether there exists a common graph with chromatic number at least four.
While odd cycles~\cite{Sid89} and even wheels~\cite{JagST96,Sid96}
are examples of $3$-chromatic common graphs, also see~\cite{GrzLLV22},
the existence of a common graph with chromatic number at least four was open until 2012
when the $5$-wheel was shown to be common~\cite{HatHKNR12} as
one of the first application of the flag algebra method of Razborov~\cite{Raz07}.
The question whether there exist common graphs with arbitrarily large chromatic number has been reiterated in~\cite{HatHKNR12}, and also by
Conlon, Fox and Sudakov 
in the survey paper ``Recent developments in graph Ramsey theory'',
which is now a classical reference for results and directions related to Ramsey theory.

\begin{problem*}[{Conlon, Fox and Sudakov~\cite[Problem 2.28]{ConFS15}}]
Do there exist common graphs of all chromatic numbers?
\end{problem*}

In this paper, we solve this problem by establishing the following.

\begin{theorem}
\label{thm:main}
For every $\ell\in\NN$, there exists a connected common graph with chromatic number $\ell$.
\end{theorem}

We treat common graphs using methods from the theory of graph limits
while employing spectral tools from the operator theory;
graph limit methods permit avoiding lower order terms when analyzing large graphs,
which makes our arguments simpler to present and in some cases also more natural.
To prove Theorem~\ref{thm:main},
we show that every graph of sufficiently large girth can be embedded in a graph $H$ such that
any $2$-edge-coloring of a complete graph
has asymptotically at least as many monochromatic copies of $H$ as the random $2$-edge-coloring.
The proof of Theorem~\ref{thm:main} is split into two cases based
on whether the considered $2$-edge-coloring is close to the random coloring or not;
we refer to the two cases as the local regime and the non-local regime (we provide an overview of the proof
in Subsection~\ref{subsec:overview}).
The core of the proof is formed by the arguments related to the local regime,
which is described by the existence of a dominant eigenvalue of the operator associated with the $2$-edge-coloring.
While both the Sidorenko property and common graphs have been studied in the local regime~\cite{Lov11,FoxW17,FoxW,CsoHL23,HanKKV23},
the proof of Theorem~\ref{thm:main} requires new spectral arguments
to control (in)dependence of monochromatic embeddings of different parts of $H$ in the host edge-colored complete graph.
Our techniques could be used in other settings,
e.g., the setting of a conjecture of Kohayakawa, Nagle, R\"odl and Schacht~\cite{KohNRS10}
as discussed in Section~\ref{sec:concl}.

Our techniques also apply in the setting of $k$-common graphs introduced in~\cite{JagST96}:
for a positive integer $k$, a graph $H$ is \emph{$k$-common}
if the random $k$-edge-coloring of a complete graph asymptotically minimizes
the number of monochromatic copies of $H$ among all $k$-edge-colorings.
This notion provides another link to the Sidorenko property:
a graph $H$ has the Sidorenko property if and only if
the graph $H$ is $k$-common for all $k\ge 2$~\cite{KraNNVW22}.
If $H$ is $k$-common, then $H$ is $k'$-common for all $k'=2,\ldots,k$, and
thus $k$-common graphs for $k\ge 3$ are even more rare than common graphs (as such graphs are necessarily also common).
In fact, the question of Jagger, \v S\v tov\'\i\v cek and Thomason~\cite{JagST96}
about the existence of a non-bipartite $k$-common graph for $k\ge 3$
has been resolved only recently in~\cite{KraNNVW22}.
Using the techniques developed to prove Theorem~\ref{thm:main}, we also prove the following.

\begin{theorem}
\label{thm:maink}
For every integer $k\ge 2$ and positive integer $\ell$,
there exists a connected $k$-common graph with chromatic number $\ell$.
\end{theorem}

\subsection{Proof motivation}
\label{subsec:motivation}

To motivate our proof, we sketch a possible strategy for proving the existence of
a \emph{disconnected} common graph with high chromatic number.
We present the strategy in the language of graph limits,
which is reviewed in Section~\ref{sec:notation}.
Fix a graph $H$ with girth four and chromatic number $k\ge 3$ and
consider the disjoint union of the graph $H$ and a complete bipartite graph, denoted as $H\cup K_{n,n}$, for a large $n$. 
Showing that $H\cup K_{n,n}$ is common amounts to showing that
\begin{equation}
t(H\cup K_{n,n},W)+t(H\cup K_{n,n},1-W)\ge 2^{-\nume{H}-n^2+1}
\label{eq:overview}
\end{equation}
for every graphon $W$, where $\nume{H}$ denotes the number of edges in $H$.
Note that the right side of \eqref{eq:overview} is equal to the left side when $W$ is the constant graphon equal $1/2$,
i.e., $W$ is the limit graphon for the sequence of the Erd\H os-R\'enyi random graphs $G_{m,1/2}$.
Since the inequality \eqref{eq:overview} is symmetric in $W$ and $1-W$,
it is enough to establish \eqref{eq:overview} for graphons $W$ with density at least $1/2$,
i.e., when $\int_{[0,1]^2} W\ge 1/2$.

The graph $H\cup K_{n,n}$ is locally common in the following sense as shown by Fox and the last author~\cite{FoxW17,FoxW}.
\begin{theorem}[Fox and Wei~\cite{FoxW17}]
\label{thm:girtheven}
If a graph $G$ has even girth, then for any $p\in [0,1]$, any graphon $W$ with density $p$ such that
$\cut{W-p}\le p^2 2^{-48\nume{G}-2}$, and $\|W\|_\infty\le 2p$ 
satisfies that $t(G,W)\ge p^{\nume{G}}$.
\end{theorem}
\noindent Hence, if a graphon $W$ is close (in the cut distance $\cut{\cdot}$) to the $1/2$-graphon,
then Theorem~\ref{thm:girtheven} almost implies that \eqref{eq:overview};
the issue lies in the fact that if $W$ has density $p>1/2$,
then the graphon $1-W$, whose density is $1-p$,
may fail to satisfy $\|1-W\|_\infty\le 2(1-p)$ needed to apply Theorem~\ref{thm:girtheven}.
The proof of Theorem~\ref{thm:girtheven} in~\cite{FoxW}
can be extended, at the expense of having a stronger bound on $\cut{W-p}$,
to the setting with the condition $\|W\|_\infty\le 2p$ weakened to $\|W\|_\infty\le Cp$ for any fixed $C\in\NN$,
however, the bound on $\cut{W-p}$ would depend on $H$ (in addition to $C$).
Such an extension of Theorem~\ref{thm:girtheven} would be sufficient for the case of the disconnected graph $H\cup K_{n,n}$,
however,
the proofs of Theorems~\ref{thm:pair} and~\ref{thm:step} in Section~\ref{sec:main} require
a bound on $\cut{W-p}$ that is independent of the actual graph $H$ for a specific family of graphs $H$.
Indeed, the value of $\varepsilon_0$ in Theorem~\ref{thm:core},
which gives a variant of Theorem~\ref{thm:girtheven} tailored for our setting,
does not depend on $\ell$.
Without removing the dependance on $H$ (or at least proving a significantly better dependance than in Theorem~\ref{thm:girtheven}),
which is the crucial contribution of the arguments presented in Section~\ref{sec:locSidorenko},
it would not be possible to set up the parameters to prove Theorems~\ref{thm:pair} and~\ref{thm:step} and
so Theorems~\ref{thm:main} and~\ref{thm:maink}.

We now continue a sketch of the argument that $H\cup K_{n,n}$ is common.
As discussed, if $W$ is close to the $1/2$-graphon,
then the inequality \eqref{eq:overview} holds by a suitable extension of Theorem~\ref{thm:girtheven}.
If a graphon $W$ is not close (in the cut distance) to the $1/2$-graphon,
we can use that the graph $K_{n,n}$ has the Sidorenko property
in the strong quantitative sense as given by Lemmas~\ref{lm:gammagamma} and~\ref{lm:K2n2nC4a}:
\[t(K_{n,n},W)\ge\left(p^4+\frac{\cut{W-p}^4}{8}\right)^{n^2/4}.\]
Hence, the quantity $t(K_{n,n},W)$ is significantly larger than $p^{n^2}$,
which can balance out a possible drop of the density of $H$ in $W$ and
so implies that
\[t(H\cup K_{n,n},W)\ge 2^{-\nume{H}-n^2+1}.\]
However, if a graphon $W$ contains a sparse part,
it may even happen that $t(H,W)=0$, so $t(H\cup K_{n,n},W)=0$, and
the argument fails (in fact,
it can be shown that the presence of a sparse part is essentially the only obstacle for the argument to proceed).
However,
if $W$ contains a sparse part of measure $\alpha_0>0$ with density $q$ close to $0$,
then this part has density $1-q$ in $1-W$ and
its contribution to $t(H\cup K_{n,n},1-W)$ (as formalized in Lemma~\ref{lm:omegaalpha})
is roughly $\alpha_0^{\numv{H}+2n}(1-q)^{\nume{H}+n^2}$;
the value of $\alpha_0$ depends on $H$ only and so if $n$ is sufficiently large,
then \[\alpha_0^{\numv{H}+2n}(1-q)^{\nume{H}+n^2}\ge 2^{-\nume{H}-n^2+1}.\]
Here,
we use that the number of edges of $H \cup K_{n,n}$
grows superlinearly in the number of vertices of $H\cup K_{n,n}$ with $n$ tending to infinity.

\subsection{Proof overview}
\label{subsec:overview}

On a high level,
the proofs of Theorems~\ref{thm:pair} and~\ref{thm:step},
which imply Theorems~\ref{thm:main} and~\ref{thm:maink}, respectively,
follows the lines described in Subsection~\ref{subsec:motivation};
the main steps of the two proofs are visualized in Figure~\ref{fig:overview}.

\begin{figure}
\begin{center}
\begin{tabbing}
XXXX \= XXXX \= XXXX \kill
Theorem~\ref{thm:core} (local regime)\\
{\it $W$ with density $p\ge p_0$ and $\cut{W-p}\le\varepsilon$ $\Rightarrow$ $t(H_0,W)\ge p^{e(H_0)}$}\\
\\
Theorem~\ref{thm:bip} (non-local regime)\\
{\it $W$ with density $p\ge p_0$, no sparse part and $\cut{W-p}\ge\varepsilon$ $\Rightarrow$ $t(H_0,W)\ge p^{e(H_0)}$}\\
\\
Theorem~\ref{thm:pair}\\
{\it $H_0$ is common, i.e., $t(H_0,W)+t(H_0,1-W)\ge 2\cdot 2^{-e(H_0)}$}\\
\\
\> $\bullet$ $W$ has a part with measure $\alpha_0$ and density $q$ close to zero \\
\> \> $\Rightarrow$ $t(H_0,1-W)\ge\alpha_0^{v(H_0)}(1-q)^{e(H_0)}\ge 2\cdot 2^{-e(H_0)}$\\
\> $\bullet$ $1-W$ has a part with measure $\alpha_0$ and density $q$ close to zero \\
\> \> $\Rightarrow$ $t(H_0,W)\ge\alpha_0^{v(H_0)}(1-q)^{e(H_0)}\ge 2\cdot 2^{-e(H_0)}$\\
\> $\bullet$ $W$ has density $\ge 1-p_0$\\
\> \> Theorem~\ref{thm:core} or Theorem~\ref{thm:bip} $\Rightarrow$ $t(H_0,W)\ge (1-p_0)^{e(H_0)}\ge 2\cdot 2^{-e(H_0)}$\\
\> $\bullet$ $1-W$ has density $\ge 1-p_0$\\
\> \> Theorem~\ref{thm:core} or Theorem~\ref{thm:bip} $\Rightarrow$ $t(H_0,1-W)\ge (1-p_0)^{e(H_0)}\ge 2\cdot 2^{-e(H_0)}$\\
\> $\bullet$ $W$ has density $p\in [p_0,1-p_0]$ and no sparse part \\
\> \> Theorem~\ref{thm:core} or Theorem~\ref{thm:bip} $\Rightarrow$ $t(H_0,W)\ge p^{e(H_0)}$\\
\> \> Theorem~\ref{thm:core} or Theorem~\ref{thm:bip} $\Rightarrow$ $t(H_0,1-W)\ge (1-p)^{e(H_0)}$\\
\> \> $t(H_0,W)+t(H_0,1-W)\ge p^{e(H_0)}+(1-p)^{e(H_0)}\ge 2\cdot 2^{-e(H_0)}$\\
\\
Theorem~\ref{thm:step}\\
{\it $H_0$ is $k$-common, i.e., $t(H_0,W_1)+\cdots t(H_0,W_k)\ge k\cdot k^{-e(H_0)}$}\\
\\
\> $\bullet$ $\exists i$ $W_i$ has a part with measure $\alpha_0$ and density close to zero \\
\> \> induction $\Rightarrow$ $t(H_0,W_1)+\cdots+t(H_0,W_k)\gtrapprox\alpha_0^{v(H_0)}(k-1)^{-e(H_0)}\ge k\cdot k^{-e(H_0)}$\\
\> $\bullet$ $\exists i$ $W_i$ has density smaller than $p_0$ $\Rightarrow$ $\exists j$ density of $W_j$ is at least $(k-1)^{-1}$\\
\> \> $\Rightarrow$ $t(H_0,W_j)\gtrapprox (k-1)^{-e(H_0)}\ge k\cdot k^{-e(H_0)}$\\
\> $\bullet$ Each $W_i$ has density $p_i\ge p_0$ bounded away from zero and no sparse part \\
\> \> Theorem~\ref{thm:core} or Theorem~\ref{thm:bip} $\Rightarrow$ $t(H_0,W_i)\ge p_i^{-e(H_0)}$\\
\> \> $t(H_0,W_1)+\cdots+t(H_0,W_k)\ge p_1^{e(H_0)}+\cdots+p_k^{e(H_0)}\ge k\cdot k^{-e(H_0)}$
\end{tabbing}
\end{center}
\caption{Informal statements of Theorems~\ref{thm:core} and~\ref{thm:bip}, and
         the main steps in the proofs of Theorems~\ref{thm:pair} and~\ref{thm:step},
	 which assert that a suitable graph $H_0$ is common and $k$-common, respectively.}
\label{fig:overview}
\end{figure}

In the proof of Theorem~\ref{thm:pair},
we show that a graph $H_0$ obtained by joining a suitable high-girth high-chromatic chromatic graph $H$
by a suitably long path to a complete bipartite graph $K_{m,n}$ is common.
The length $\ell$ of the path joining $H$ and the complete bipartite graph needs to be carefully controlled.
On one hand,
the path needs to be long enough to make copies of $H$ and $K_{m,n}$ in a graphon $W$ sufficiently independent in the sense 
that $t(H_0,W)$ is approximately $t(H,W)t(K_{m,n},W)p^{\ell}$ for any graphon $W$ with density $p$;
note that $t(H\cup K_{m,n},W)=t(H,W)t(K_{m,n},W)$.
At the same time, the number of edges of the graph $H_0$ needs to stay superlinear in the number of the vertices of $H_0$,
i.e., the length $\ell$ needs to be $o(v(H)+mn)$, so that
we are able to handle the case when the graphon $W$ contains a sparse part of measure $\alpha_0>0$ with density $q$ close to $0$.
In such case,
the density of $H_0$ in $W$ can be zero but
the density of $H_0$ in $1-W$ is approximately $\alpha_0^{v(H_0)}(1-q)^{e(H_0)}$,
which needs to be at least $2^{-e(H_0)+1}$ in order to yield that $H_0$ is common.

As the first step towards proving Theorems~\ref{thm:pair} and~\ref{thm:step},
we deal with the \emph{local regime} and
establish a counterpart of Theorem~\ref{thm:girtheven},
which is given as Theorem~\ref{thm:core} in Section~\ref{sec:locSidorenko}.
While Theorem~\ref{thm:core} applies to graphs of a particular form only,
the bound on the cut distance given in Theorem~\ref{thm:girtheven}
is independent of the length $\ell$ of the path attached to the graph $H$ (in addition
to relaxing the technical assumption that $\|W\|_\infty\le 2p$ as discussed in Subsection~\ref{subsec:motivation}).
The independence on $\ell$ is essential for proving Theorem~\ref{thm:universal} in Section~\ref{sec:local},
where the obtained bound on the cut distance has to be independent of the parameters $\ell$, $m$ and $n$ so that
it is sufficiently strong to eventually yield proofs of Theorems~\ref{thm:pair} and~\ref{thm:step}.
The proof of Theorem~\ref{thm:universal}, which is the technically most challenging part of our argument,
is obtained using spectral arguments,
which allow us to control how densities of rooted subgraphs change when walking along a path in a graphon.
In particular,
we argue that the density of $H_0$ in a graphon $W$
is at least the density of $H_0$ in the constant graphon with the same density as $W$ (Cases \ref{it:1} and \ref{it:2} in the proof)
unless the graphon $W$ has a sparse part (Case \ref{it:3}),
which in turn implies that it also has a denser part
where sufficiently many copies of the graph $H_0$ can be found.

The \emph{non-local regime} is addressed in Section~\ref{sec:nonlocal}.
First, Lemma~\ref{lm:omegaalpha} yields that 
the density of a fixed graph $H$ in any graphon $W$ is positive unless $W$ contains a sparse part,
however, the constants in the statement of the lemma depend on $H$.
Lemma~\ref{lm:omegaalpha} is then used to prove Theorem~\ref{thm:bip} that asserts that
if a graphon $W$ is far in the cut distance from the constant (quasirandom) graphon,
then the density of the graph $H_0$ in $W$ is large enough
unless $W$ contains a sparse part whose size is independent of the parameters $\ell$, $m$ and $n$ (from a suitably chosen range).

Finally, Section~\ref{sec:main} presents our main results---Theorems~\ref{thm:pair} and~\ref{thm:step};
the main steps of their proofs are visualized in Figure~\ref{fig:overview}.
If a graphon $W$ or $1-W$ has a sparse part,
then the density of $H_0$ in the complement of the sparse part is sufficiently large to establish that $H_0$ is common;
here, we need that $\ell$ is sublinear in $v(H)+mn$.
Otherwise, if a graphon $W$ is close to a constant graphon,
then the density of $H_0$ in both $W$ and $1-W$
is at least the density of $H_0$ in the constant graphons with the same density by Theorem~\ref{thm:universal}, and
if $W$ is not close to a constant graphon (and does not have a sparse part),
then the density of $H_0$ in both $W$ and $1-W$
is at least the density of $H_0$ in the constant graphons with the same density by Theorem~\ref{thm:bip}.
A simple convexity argument now yields that $t(H_0,W)+t(H_0,1-W)\ge 2\cdot 2^{-e(H_0)}$.

The proof of the more general Theorem~\ref{thm:step},
which concern $k$-common graphs, proceeds by induction on $k$,
which is the number of colors.
Individual color classes are represented by graphons $W_1,\ldots,W_k$ such that
$W_1+\cdots+W_k$ is a graphon with density close to one (rather than exactly equal to one);
this makes the induction argument significantly easier.
If one of graphons $W_1,\ldots,W_k$  has a very sparse part,
we apply induction to the remaining $k-1$ colors restricted to the sparse part.
If one of the color classes is sparse,
we argue that we find sufficiently many copies of $H_0$ in the densest color class (here,
we use that $\ell$ is sublinear in $v(H)+mn$).
If neither of these two cases apply,
we use for each graphon $W_i$, $i=1,\ldots,k$, Theorem~\ref{thm:universal} or Theorem~\ref{thm:bip},
depending whether $W_i$ is close to a constant graphon or not,
i.e., whether the corresponding color class falls into the local or non-local regime,
to derive that the density of $H_0$ in each $W_i$ is at least as in the constant graphon with the same density, and
we eventually derive using convexity of the function $p^{e(H_0)}$ that
the sum of densities of $H_0$ in $W_1,\ldots,W_k$ is sufficiently large.

\section{Preliminaries}
\label{sec:notation}

In this section, we fix notation used throughout the paper and
present auxiliary results needed in our arguments.
We start with basic notation.
The set of the first $k$ positive integers is denoted by $[k]$, and
the set of all $k$-element subsets of $A$ is denoted by $\binom{A}{k}$.
If $A$ is a measurable subset of $\RR^k$, we write $\mu(A)$ for the measure of $A$;
throughout the paper, we always consider the standard Borel measure on $\RR^k$.

All graphs considered in this paper are simple and loopless.
If $G$ is a graph, then the vertex set of $G$ is denoted by $V(G)$ and the edge set by $E(G)$;
the numbers of vertices and edges of $G$ are denoted by $\numv{G}$ and $\nume{G}$, respectively.
If $A$ is a subset of vertices of $G$, then $G[A]$ is the subgraph induced by $A$,
i.e., the graph with the vertex set $A$ and the edge set $E(G)\cap\binom{A}{2}$.
Finally, if $F$ is a subset of edges of $G$,
then $\langle F\rangle$ is the spanning subgraph of $G$ with the edge set $F$;
the host graph $G$ will always be clear from the context.
A \emph{homomorphism} from a graph $H$ to $G$ is mapping $f$ from $V(H)$ to $V(G)$ such that
if $uv$ is an edge of $H$, then $f(u)f(v)$ is an edge of $G$, and
the \emph{homomorphism density} of $H$ in $G$, denoted by $t(H,G)$,
is the probability that a random mapping from $V(H)$ to $V(G)$ is a homomorphism.

The $n$-vertex path is denoted by $P_n$, the $n$-vertex cycle by $C_n$,
the complete graph with $n$ vertices by $K_n$, and
the complete bipartite graph with parts consisting of $a$ and $b$ vertices by $K_{a,b}$.
Finally, $K_{a|\ell,b}$ is the graph obtained from $K_{a,b}$
by adding an $\ell$-edge path to one of the vertices in the $a$-vertex part.
In particular, $K_{a|0,b}$ is just the graph $K_{a,b}$.

\subsection{Graph limits}
\label{subsec:limits}

We treat the notion of common graphs using tools from the theory of graph limits.
In this subsection, we provide an introduction to the most important concepts;
we refer the reader to the monograph by Lov\'asz~\cite{Lov12} for a more complete treatment.
In the theory of graph limits, large graphs are represented by an analytic object called graphon.
A~\emph{graphon} is a measurable function $W:[0,1]^2\to [0,1]$ that is symmetric,
i.e., $W(x,y)=W(y,x)$ for all $x,y\in [0,1]$.
When no confusion can arise, we use $p$ to denote the graphon equal to $p\in [0,1]$ everywhere.
More generally, a \emph{kernel} is a symmetric measurable function $U:[0,1]^2\to\RR$.
In particular, a graphon is a kernel with range in $[0,1]$.
A graphon can be thought of as a continuous version of the adjacency matrix of a graph.
Because of this analogy, we refer to the elements of the domain $[0,1]$ as to vertices of $W$.
The \emph{degree} of a vertex $x\in [0,1]$ of a graphon $W$ is defined as
\[\deg_W(x)=\int_{[0,1]}W(x,y)\dd y.\]
The \emph{homomorphism density} of a graph $H$ in a graphon $W$ is defined as
\begin{equation}
t(H,W) = \int_{[0,1]^{V(H)}} \prod_{uv\in E(H)} W(x_u,x_v) \dd x_{V(H)}.\label{eq:tHW}
\end{equation}
For brevity, we often speak about the density of $H$ in $W$ instead of the homomorphism density of $H$ in $W$.
The \emph{density} of a graphon $W$ is the density of $K_2$ in $W$, i.e., $t(K_2,W)$.
Analogously, we define the homomorphism density of a graph $H$ in a kernel $U$ using \eqref{eq:tHW}, and
we define the density of a kernel $U$ to be $t(K_2,U)$.

We now cast the definition of graphs with the Sidorenko property and commons graphs in the language of graph limits,
see e.g.~\cite{KraNNVW22} for further details.
A graph $H$ has the \emph{Sidorenko property} if the following holds for every graphon $W$:
\[p^{\nume{H}}\le t(H,W)\]
where $p$ is the density of $W$.
A graph $H$ is \emph{common} if it holds that
\[2^{1-\nume{H}}\le t(H,W)+t(H,1-W)\]
for every graphon $W$.
Finally, for an integer $k\ge 2$,
we say that a graph $H$ is \emph{$k$-common} if
\[k^{1-\nume{H}}\le t(H,W_1)+\cdots+t(H,W_k)\]
holds for all graphons $W_1,\ldots,W_k$ such that $W_1+\cdots+W_k=1$.

We next define the notion of a norm for kernels,
which gives a metric on the space of kernels and so on graphons.
The \emph{cut norm} of a kernel $U$ is defined as
\[\cut{U} = \sup_{S, T \subseteq [0,1]} \left\lvert \int_{S \times T} U(x,y) \dd x \dd y\right\rvert\]
where the supremum is taken over all measurable subsets $S$ and $T$ of $[0,1]$.
If $U$ is a kernel with $\|U\|_{\infty}\le 1$,
then it holds~\cite[Lemma 8.12]{Lov12} that
\begin{equation}
\cut{U}^4\le t(C_4,U)\le 4\cut{U}.\label{eq:cutU}
\end{equation}
Similarly, it holds that $t(P_2,U)\le 2\cut{U}$.
Since $C_4$ has the Sidorenko property, it holds that $t(C_4,W)\ge p^4$ for every graphon $W$ with density $p$.
If a graphon $W$ with density $p$ is actually far from the $p$-constant graphon,
then the density of $C_4$ is much larger than $p^4$ as given in the lemma.
\begin{lemma}[{Cooper, Kr\'{a}\v{l} and Martins~\cite[Lemma 11]{CooKM18}}]
\label{lm:gammagamma}
The following holds for every graphon $W$ with density $p$:
\[p^4+\frac{\cut{W-p}^4}{8}\le t(C_4,W).\]
\end{lemma}
On the other hand,
the Counting Lemma, which we now state,
implies that $t(C_4,W)\le p^4+4\cut{W-p}$ for every graphon $W$ with density $p$.
\begin{lemma}[Counting Lemma, {Lov\'asz~\cite[Lemma 10.23]{Lov12}}]
\label{lm:cutdistance}
The following holds for every graph $H$ and all graphons $W_1$ and $W_2$:
\[
\left\lvert t(H,W_1)-t(H,W_2)\right\rvert\le\nume{H}\cdot\cut{W_1-W_2}.
\]
\end{lemma}

We next define a notion used in~\cite{KraNNVW22} that can be viewed as an analogue of a subgraph in a graphon.
Consider a measurable function $h:[0,1]\to [0,1]$ such that $\|h\|_1>0$.
Let $f:[0,\|h\|_1]\to [0,1]$ be the measurable function defined as
\[f(z) := \inf \left\{t\in [0,1]\mbox{ such that }\int_{[0,t]} h(x)\dd x\ge z\right\}.\]
Note that it holds 
$\int_A h(x)\dd x=\mu(f^{-1}(A))$
for every measurable subset $A \subseteq [0,1]$.
We define the graphon $W[h]$ by setting
\[W[h](x,y)=W\left(f(x\cdot \|h\|_1),\;f(y\cdot \|h\|_1)\right)
 \quad
 \mbox{for $(x,y)\in [0,1]^2$.}
 \]
Note that if $h$ is the indicator function of a measurable subset $A$,
then $W[h]$ is obtained by restricting $W$ to $A$ and rescaling;
if this is the case, we write $W[A]$ for the graphon $W[h]$.
It can be shown that
\[
t(H,W[h])=\frac{1}{\|h\|_1^{\numv{H}}}\int_{[0,1]^{V(H)}} \prod_{u\in V(H)}h(u) \prod_{uv\in E(H)} W(x_u,x_v) \dd x_{V(H)}
\]
holds for every graph $H$.
In particular, it holds that
\begin{equation}
t(H,W)\geq \|h\|_1^{\numv{H}}\cdot t(H,W[h]).\label{eq:tWh}
\end{equation}
We conclude this subsection by defining a notion of an almost independent set in a graphon.
For a graphon $W$ and $\delta>0$,
$\AAA_{\delta}(W)$ is the set of all measurable functions $h:[0,1]\to [0,1]$ with $\|h\|_1>0$ such that
the density of $W[h]$ is at most $\delta$.
The \emph{$\delta$-independence ratio} of $W$ is defined as
\[\alpha_{\delta}(W)=\sup_{h\in\AAA_{\delta}(W)}\|h\|_1;\]
if the set $\AAA_{\delta}(W)$ is empty,
then we set $\alpha_{\delta}(W)=0$.
We remark that this notion is closely related to the notion of $(\rho,d)$-dense graphs widely used in Ramsey Theory:
a graph is \emph{$(\rho,d)$-dense}
if any subset of at least $\rho\cdot\numv{G}$ vertices of $G$ induces a subgraph with density at least $d$.
Indeed, if a graphon $W$ is a limit graphon of a sequence of $(\rho,d)$-dense graphs,
then $\alpha_{\delta}(W)<\rho$ for every $\delta<d$.

\subsection{Rooted graphs}

A \emph{rooted graph} is a graph with one or more vertices distinguished as \emph{roots}.
We will use superscripts to emphasize that a graph is rooted and indicate the number of roots.
In particular, $H^\bullet$ will denote a rooted graph with a single root,
$H^{\bullet\bullet}$ will denote a rooted graph with two roots, and
$H^{\bullet\cdots\bullet}$ will denote a rooted graph with three or more roots.
In all rooted graphs considered in this paper,
the set of roots will always form an \emph{independent set}.
To simplify our notation, if $H^\bullet$ is a rooted graph,
$\numv{H}$ and $\nume{H}$ denote the number of vertices and edges of $H^\bullet$, respectively, as
these quantities are the same regardless of the presence or the absence of roots.
Examples of rooted graphs considered further include the following.
The rooted graph $P_{n}^{\bullet}$
is the graph obtained from the $n$-vertex path by choosing one of its end vertices as a root.
The rooted graph $K_{a,b}^{\bullet}$
is the graph obtained from the complete bipartite graph $K_{a,b}$ by choosing one of the vertices of the $a$-vertex part as a root.
Finally,
the rooted graph $K_{a|\ell,b}^{\bullet}$ 
is the graph obtained from $K_{a|\ell,b}$ by choosing the end vertex of the appended $\ell$-edge path as a root.

If $G^\bullet$ and $H^\bullet$ are two rooted graphs,
then $G^\bullet\oplus H^\bullet$ is the (unrooted) graph obtained by identifying their roots.
Note that $K_{a|\ell,b}=P_{\ell+1}^\bullet\oplus K_{a,b}^\bullet$.
Similarly, if $G^{\bullet\bullet}$ and $H^{\bullet\bullet}$ are two rooted graphs with two roots,
then $G^{\bullet\bullet}\oplus H^{\bullet\bullet}$ is the graph obtained by identifying the corresponding roots, and
if $G^{\bullet\cdots\bullet}$ and $H^{\bullet\cdots\bullet}$ are two rooted graphs with the same number of roots,
then $G^{\bullet\cdots\bullet}\oplus H^{\bullet\cdots\bullet}$ is the graph obtained by identifying the corresponding roots;
the correspondence will always be clear from the context.

We next extend the notion of homomorphism density to rooted graphs.
Let $H^{\bullet\cdots\bullet}$ be a rooted graph with $k$ roots, and
let $v_1,\ldots,v_n$ be the vertices of $H^{\bullet\cdots\bullet}$ listed in a way that
the vertices $v_1,\ldots,v_k$ are the roots.
If $U$ is a kernel and $x_1,\ldots,x_k\in [0,1]$,
we define $t_{x_1,\ldots,x_k}(H^{\bullet\cdots\bullet},U)$ as
\begin{equation}
t_{x_1,\ldots,x_k}(H^{\bullet\cdots\bullet},U)=
\int_{[0,1]^{n-k}}\prod_{v_iv_j\in E(H^{\bullet\cdots\bullet})}U(x_i,x_j)\dd x_{k+1}\cdots\dd x_n.
\label{eq:tHUrooted}
\end{equation}
Observe that if $W$ is a graphon, then
$\deg_W(x)=t_x(P_2^\bullet,W)$
for $x\in [0,1]$.
Finally, if $G^{\bullet\cdots\bullet}$ and $H^{\bullet\cdots\bullet}$ are two rooted graphs such that
each has $k$ roots and the roots induce an independent set,
it holds that
\begin{equation}
t\left(G^{\bullet\cdots\bullet}\oplus H^{\bullet\cdots\bullet},U\right)=
\int_{[0,1]^k}t_{x_1,\ldots,x_k}\left(G^{\bullet\cdots\bullet},U\right)
              t_{x_1,\ldots,x_k}\left(H^{\bullet\cdots\bullet},U\right)\dd x_1\cdots\dd x_k
\label{eq:tHUmerge}
\end{equation}
for every kernel $U$;
note that the assumption that
the roots of $G^{\bullet\cdots\bullet}$ and $H^{\bullet\cdots\bullet}$ form independent sets
is needed for \eqref{eq:tHUmerge} to hold.

\subsection{Spectral properties of graphons and kernels}
\label{subsec:spectral}

We now review spectral properties of graphons and kernels;
we refer to~\cite[Section 7.5]{Lov12} for further details.
Fix a kernel $U$ with $\|U\|_\infty\le 1$.
We can think of $U$ as a Hilbert-Schmidt integral operator from $L_2[0,1]$ to $L_2[0,1]$ defined as
\[(Uf)(x)= \int_{0}^1 U(x,y) f(y) \dd y.\]
Let $\lambda_1,\lambda_2,\ldots$ be the non-zero eigenvalues of $U$ listed in the non-increasing order of the absolute value (with multiplicities), and
let $f_1,f_2,\ldots$ be the corresponding orthonormal eigenfunctions,
i.e., $\|f_i\|_2=1$ for every $i$ and the functions are orthogonal to each other in $L_2[0,1]$.
Note that
\[\sum_i \lambda_i f_i(x)f_i(y)\]
converges to $U$ in the $L_2$-norm.
We remark that if $U$ is a graphon with density $p$, then $\lambda_1$ is at least $p$.
Since we have assumed that $\|U\|_\infty\le 1$,
it holds that $(Uf_i)(x)\le\|f_i\|_2=1$ for every $x\in [0,1]$,
i.e., $\|Uf_i\|_\infty\le 1$,
which implies that $\|f_i\|_{\infty}\le \lvert\lambda_i\rvert^{-1}$.

The spectrum of $U$ can be used to express the density of cycles and paths.
Expressing the density of cycles is easier:
it holds for every $n\ge 3$ that
\begin{equation}
t(C_n,U) = \sum_{i} \lambda_i^n.
\label{eq:Cn}
\end{equation}
To express the density of paths, we need to introduce additional notation.
Let $\jj:[0,1]\to [0,1]$ be the constant function equal to $1$, and
set $\alpha_i\in [0,\pi/2]$ to be the real such that $\langle \jj,f_i\rangle=\cos\alpha_i$ (we may assume without
loss of generality by replacing $f_i$ with $-f_i$ that the product is non-negative).
Further set $\delta=1-\cos\alpha_1$.
Informally speaking, for a graphon $W$, $\delta$ measures how close $W$ is to having all the degrees the same,
in particular, $\delta=0$ if and only if $\deg_W(x)=p$ for almost every $x\in [0,1]$ where $p$ is the density of $W$.
The following holds for every $n\ge 2$ (for $n=2$, the equality is a particular case of (7.21) in~\cite{Lov12}):
\begin{equation}
t(P_n,U) = \langle \jj,U^{n-1}\jj\rangle = \sum_{i} \lambda_i^{n-1}\cos^2\alpha_i.
\label{eq:Pn}
\end{equation}

In the rest of this subsection, we deal with graphons only and
estimate some of the introduced parameters.
These estimates will be used repeatedly in the proofs of Theorems \ref{thm:core} and  \ref{thm:universal}. 
Fix a graphon $W$ with density $p\in [0,1]$, and
let $\gamma=t(C_4,W)-p^4$.
Note that $\gamma\ge 0$ since the cycle $C_4$ has the Sidorenko property;
on the other hand, $\gamma\le 4\cut{W-p}$ by the Counting Lemma (Lemma~\ref{lm:cutdistance}).
Since the density of $C_4$ in $W$ is at least $\lambda_1^4$,
it follows that
\begin{equation}
\lambda_1\le p+\frac{\gamma}{4p^3}\le p+\frac{\cut{W-p}}{p^3}.
\label{eq:lambda1small}
\end{equation}
We derive from $\lambda_1\ge p$ and \eqref{eq:Cn} that
\begin{equation}
\lvert\lambda_i\rvert\le\gamma^{1/4}
 \label{eq:lambdasmall}
\end{equation} 
holds for every $i\ge 2$.
Hence, we derive from~\eqref{eq:lambdasmall} that
\begin{equation}
\sum_{i\ge 2} \lambda_i^m \le \gamma^{\frac{m-4}{4}}\sum_{i\ge 2}\lambda_i^{4} \le \gamma^{m/4}
\label{eq:powersumgamma}
\end{equation}
holds for every $m\ge 4$.
In addition,
\eqref{eq:lambdasmall} yields that it holds for every $m\ge 0$ that
\[\left\lvert\sum_{i\ge 2}\lambda_i^m\cos^2\alpha_i\right\rvert\le
  \sum_{i\ge 2}\lvert\lambda_i\rvert^m\cos^2\alpha_i\le
  \max_{i\ge 2}\lvert\lambda_i\rvert^m\cdot\sum_{i\ge 2}\cos^2\alpha_i\le
  \max_{i\ge 2}\lvert\lambda_i\rvert^m\le
  \gamma^{m/4},\]
which implies by~\eqref{eq:Pn} that
\begin{equation}
\lambda_1^{m}(1-\delta)^2-\gamma^{m/4}\le
t(P_{m+1},W)\le
\lambda_1^{m}(1-\delta)^2+\gamma^{m/4}.
\label{eq:Pmrange}
\end{equation}

We next estimate $\delta$ in terms of $\gamma$ and subsequently in terms of $\cut{W-p}$.
Since the path $P_5$ has the Sidorenko property and \eqref{eq:Pn} holds,
we obtain that
\begin{align*}
  p^4 \le t(P_5,W) 
      & = \sum_i \lambda_i^{4}\cos^2\alpha_i 
        \le \lambda_1^4(1-\delta)^2 + \sum_{i\ge 2}\lambda_i^4 
        = \lambda_1^4(1-\delta)^2-\lambda_1^4+\sum_i\lambda_i^4 \\
      & \le p^4(1-\delta)^2-p^4+\sum_i\lambda_i^4
        = p^4(1-\delta)^2+\gamma
        \le p^4(1-\delta)+\gamma.
\end{align*}      
It follows that
\begin{equation}
\delta\le\frac{\gamma}{p^4}\le\frac{4\cut{W-p}}{p^4}.
\label{eq:delta}
\end{equation}
Since the eigenfunctions $f_i$'s are orthonormal, we obtain that
\begin{equation}
\sum_i\cos^2\alpha_i=\sum_{i}\langle \jj, f_i\rangle^2\le\|\jj\|_2^2=1. \label{eq:sumai}
\end{equation}
The inequality~\eqref{eq:sumai} implies that
\begin{equation}
\sum_{i\ge 2}\cos^2\alpha_i\le 1 - (1-\delta)^2  = 2\delta-\delta^2\le 2\delta.
\label{eq:sumcosa}
\end{equation} 
In particular, $\cos\alpha_i\le 2\delta^{1/2}$ for every $i\ge 2$.
Using~\eqref{eq:Pn} for $n=2$, \eqref{eq:lambdasmall} and \eqref{eq:sumcosa}, we obtain that
\[p=\sum_i\lambda_i\cos^2\alpha_i\le\lambda_1(1-\delta)^2+\gamma^{1/4}\sum_{i\ge 2}\cos^2\alpha_i
    \le\lambda_1(1-\delta)^2+2\delta\gamma^{1/4}.\]
This implies for $\delta\in [0,1]$ that 
\begin{equation}
\lambda_1 \; \ge \; \frac{p-2\delta\gamma^{1/4}}{(1-\delta)^2}
          \; \ge \; (p-2\delta\gamma^{1/4})(1+\delta)^2
	  \; \ge \; p(1+2\delta) - 2\delta\gamma^{1/4}(1+\delta)^2
	  \; \ge \; p(1+2\delta) - 8\delta\gamma^{1/4}.
\label{eq:lambda1big0}
\end{equation}

\subsection{Estimates on densities}
\label{subsec:estimate}

We now review several estimates on densities of graphs in graphons and kernels.
A large number of the estimates that we present is standard and
we state them explicitly just for referencing in the rest of the paper.
We start with one of such estimates, which follows by the standard use of Jensen's inequality,
see e.g.~\cite[Proposition 1.10]{KraNNVW22}.
The two lemmas are also implied by the weak H\"older property that
all complete bipartite graphs have since they are weakly norming~\cite{Hat10},
also see~\cite[Section 14.1]{Lov12} for in-depth discussion.
\begin{lemma}
\label{lm:K2n2nC4b}
The following holds for every graphon $W$ and all integers $m,m',n$ such that $m\geq m'$:
\[t(K_{m,n},W)\geq t(K_{m',n},W)^{m/m'}.\]
\end{lemma}
Two applications of Lemma~\ref{lm:K2n2nC4b} yield the following.
\begin{lemma}
\label{lm:K2n2nC4a}
The following holds for every graphon $W$ and all integers $m, n \geq 2$:
\[ t(K_{m,n}, W) \geq t(C_4, W)^{mn/4}.\]
\end{lemma}

The estimates in the rest of the subsection concern kernels.
We start with a well-known estimate,
which follows by a straightforward application of the Cauchy-Schwarz inequality.
\begin{lemma}
\label{lm:L1}
The following holds for any kernel $U$:
\begin{equation}
0 \leq t(K_{1,2}, U) \leq t(C_4, U)^{1/2}.\label{eq:P3C4}
\end{equation}  
\end{lemma}
The next lemma follows by using Jensen's inequality and the following two identities:
\[t(K_{1,k},U)=\int_{[0,1]} t_x(P_2^{\bullet},U)^k\dd x\quad\mbox{and}\quad
  t(K_{2,k},U)=\int_{[0,1]^2} t_{x,y}(K_{2,1}^{\bullet\bullet},U)^k\dd x\dd y.\]
\begin{lemma}
\label{lm:L2}
The following holds for any kernel $U$ with $\|U\|_\infty\le 1$ and every integer $k\ge 2$:
\[\lvert t(K_{1,k},U)\rvert\le t(K_{1,2},U)^{k/2}\le t(K_{1,2},U)\quad\mbox{and}\quad \lvert t(K_{2,k},U)\rvert\le t(C_4,U)^{k/2}\le t(C_4,U).\]
\end{lemma}

The remaining estimates that we present were proven by Lov\'asz in~\cite{Lov11};
we state them with corresponding references to the statements in the paper~\cite{Lov11},
which contain them or imply them.
We remark that the second inequality in Lemma~\ref{lm:L1b} follows using Lemma~\ref{lm:L1}.

\begin{lemma}[{Lov\'asz~\cite[Lemma 3.3]{Lov11}}]
\label{lm:CS}
Let $f_1,\dots, f_n: [0,1]^k \to \mathbb{R}$ be bounded measurable functions such that
for each variable there are at most two functions $f_i$ that depend on that variable.
It holds that
\[\int_{[0,1]^k} \prod_{i\in [n]}f_i(x_1,\ldots,x_n)\dd x_1\cdots\dd x_k\leq\prod_{i\in [n]}\|f_i\|_2.\]
\end{lemma}

\begin{lemma}[{Lov\'asz~\cite[Corollary 3.12]{Lov11}}]
\label{lm:L1a}
The following holds for any kernel $U$ with $\|U\|_\infty\le 1$ and every integer $k\ge 2$:
\[t(C_{2k}, U) \leq t(C_4, U)^{k/2}.\]
\end{lemma}

\begin{lemma}[{Lov\'asz~\cite[Lemma 3.14]{Lov11}}]
\label{lm:L1b}
The following holds for any kernel $U$ with $\|U\|_\infty\le 1$ and every integer $k\ge 1$:
\[t(P_{k+3},U)^4 \leq t(P_{3},U)^4 t(C_4,U)^{k}\le t(C_4,U)^{k+2}.\]
\end{lemma}

\begin{lemma}[{Lov\'asz~\cite[Lemma 3.19]{Lov11}}]
\label{lm:L3a}
The following holds for any kernel $U$ with $\|U\|_\infty\le 1$ and
every bipartite graph $G$ with minimum degree two and girth at least $4$ that
is not a single cycle or a complete bipartite graph:
\[\lvert t(G,U)\rvert\le t(C_4,U)^{5/4}.\]
\end{lemma}

\begin{lemma}[{Lov\'asz~\cite[Lemma 3.21]{Lov11}}]
\label{lm:L3c}
The following holds for any kernel $U$ with $\|U\|_\infty\le 1$ and
any tree $T$ that is not a star:
\[\lvert t(T,U)\rvert\le t(P_3,U)t(C_4,U)^{1/4}.\]
\end{lemma}

\begin{lemma}[{Lov\'asz~\cite[Lemma 3.22]{Lov11}}]
\label{lm:L3b}
The following holds for any kernel $U$ with $\|U\|_\infty\le 1$ and
every bipartite graph $G$ with girth at least $4$ that has exactly one vertex of degree one:
\[\lvert t(G,U)\rvert\le\frac{\left(t(C_4,U)+t(P_3,U)\right)\cdot t(C_4,U)^{1/8}}{2}.\]
\end{lemma}

\subsection{Entropy based estimates}
\label{subsec:entropy}

We next estimate the density of a graph $K_{a|\ell,b}$.
The next lemma asserts that
a graph $K_{a|\ell,b}$ has the Sidorenko property in a strong sense as $t(K_{1,2},G)\ge t(K_2,G)^2$.
Our proof is based on the entropy argument,
which has been pioneered in this context by Szegedy~\cite{Sze15}.

\begin{lemma}
\label{lm:entropy}
Let $a$, $b$ and $\ell$ be any positive even integers.
It holds that for any graph $G$, 
\[t(K_{a|\ell,b},G)\ge t(K_{1,2},G)^{(ab+\ell)/2}.\]
\end{lemma}

\begin{proof}
Fix integers $a$, $b$ and $\ell$, and a graph $G$, and
let $n$ be the number of vertices of $G$.
We say that a triple $(x_1,x_2,y)$ of vertices of $G$ is \emph{feasible} if $x_1y$ and $x_2y$ are edges of $G$.
We define several probability distributions on vertices of $G$ and their tuples.
The first distribution is the uniform distribution on $t(K_{1,2},G)n^3$ feasible triples $(x_1,x_2,y)$;
note that the entropy of this distribution is
\begin{equation}
H(x_1,x_2,y)=3\log n+\log t(K_{1,2},G).\label{eq:entropy}
\end{equation}
The second distribution is the distribution on $(a+1)$-tuples $(x_1,\ldots,x_a,y)$ that
we first choose $y$ according to the marginal distribution of the uniform distribution on feasible triples and
we then choose $x_1,\ldots,x_a$ to be neighbors of $y$ chosen uniformly and independently of each other.
Observe that the marginal distribution on $y$ and any two of the vertices $x_1,\ldots,x_a$
is the uniform distribution on feasible triples and so the following holds:
\[H(x_1,\ldots,x_a,y) = \frac{a}{2}H(x_1,x_2|y)+H(y) 
                      = \frac{a}{2}H(x_1,x_2,y)-\frac{a-2}{2}H(y) 
		      \ge \frac{a}{2}H(x_1,x_2,y)-\frac{a-2}{2}\log n.
\]
The third distribution is the distribution on $(a+b)$-tuples $(x_1,\ldots,x_a,y_1,\ldots,y_b)$ that
we first choose $x_1,\ldots,x_a$ according to the marginal distribution derived from the second distribution and
we then choose each of $y_1,\ldots,y_b$ according to the second distribution conditioned on the choice of $x_1,\ldots,x_a$.
Observe that
\begin{align*}
H(x_1,\ldots,x_a,y_1,\ldots,y_b) & = b\cdot H(y|x_1,\ldots,x_a)+H(x_1,\ldots,x_a)\\
                                 & = b\cdot H(x_1,\ldots,x_a,y)-(b-1)H(x_1,\ldots,x_a)\\
				 & \ge b\cdot H(x_1,\ldots,x_a,y)-a(b-1)\log n\\
				 & \ge \frac{ab}{2}H(x_1,x_2,y)-\frac{(a-2)b}{2}\log n-a(b-1)\log n\\
				 & = \frac{ab}{2}H(x_1,x_2,y)-\frac{3ab-2a-2b}{2}\log n.
\end{align*}
Finally, the last distribution,
which is a distribution on $(a+b+\ell)$-tuples of vertices, is obtained as follows.
We first choose $(x_1,\ldots,x_a,y_1,\ldots,y_b)$ according to the third distribution;
note that the marginal distribution of $x_1$ is the same as its marginal distribution of the uniform distribution on feasible triples.
Next choose $(x_1,y'_1,x'_1)$ uniformly among all feasible triples with the first coordinate equal to $x_1$,
then $(x'_1,y'_2,x'_2)$ uniformly among all feasible triples with the first coordinate equal to $x'_1$,
then $(x'_2,y'_3,x'_3)$ uniformly among all feasible triples with the first coordinate equal to $x'_2$
and so on until the triple $(x'_{\ell/2-1},y'_{\ell/2},x'_{\ell/2})$ has been chosen.
Observe that
\[
H(x_1,y'_1,x'_1,\ldots,y'_{\ell/2},x'_{\ell/2})  =\frac{\ell}{2} H(y,x_2|x_1)+H(x_1)
                                                 =\frac{\ell}{2} H(x_1,y,x_2)-\left(\frac{\ell}{2}-1\right)H(x_1).
\]
It follows that $H(x_1,\ldots,x_a,y_1,\ldots,y_b,y'_1,x'_1,\ldots,y'_{\ell/2},x'_{\ell/2})$ is equal to
\begin{align*}
 & H(x_1,y'_1,x'_1,\ldots,y'_{\ell/2},x'_{\ell/2})+H(x_1,\ldots,x_a,y_1,\ldots,y_b)-H(x_1) \\
 & \ge \frac{ab+\ell}{2}H(x_1,x_2,y)-\frac{3ab-2a-2b}{2}\log n-\frac{\ell}{2}H(x_1) \\
 & \ge \frac{ab+\ell}{2}H(x_1,x_2,y)-\frac{3ab-2a-2b+\ell}{2}\log n \\
 & = (a+b+\ell)\log n+\frac{ab+\ell}{2}\log t(K_{1,2},G),
\end{align*}
where the last inequality is implied by \eqref{eq:entropy}.
Since $H(x_1,\ldots,x_a,y_1,\ldots,y_b,y'_1,x'_1,\ldots,y'_{\ell/2},x'_{\ell/2})$
is at most $\log\left(t(K_{a|\ell,b},G)n^{a+b+\ell}\right)$,
we conclude that $\log t(K_{a|\ell,b},G)\ge\frac{ab+\ell}{2}\log t(K_{1,2},G)$,
i.e., $t(K_{a|\ell,b},G)\ge t(K_{1,2},G)^{(ab+\ell)/2}$ as desired.
\end{proof}

Our second estimate on the density of a graph $K_{a|\ell,b}$ is a quantitative bound
on the density of $K_{a|\ell,b}$ in a graphon $W$
depending on its distance from the constant graphon with the same density.

\begin{lemma}
\label{lm:Kab}
For all positive even integers $a$, $b$ and $\ell$ such that ${\ell} \leq ab/4$,
it holds that
\[t(K_{a|\ell,b},W)\ge p^{ab+\ell}\left(1+10^{-9}\cut{W-p}^{16}\right))^{ab/4+\ell/4}\]
for every graphon $W$ with density $p$.
\end{lemma}

\begin{proof}
Fix integers $a$, $b$ and $\ell$, and a graphon $W$ with density $p$;
set $\varepsilon=\cut{W-p}$.
Let $A$ be the set of $x\in [0,1]$ such that $\deg_W(x)\le p(1-0.001\varepsilon^5)$.
We distinguish two cases depending on the measure of $A$.

The first case is the case when $\mu(A)<0.001p\varepsilon^5$,
i.e., the degree of most of the vertices of $W$ is close to $p$.
Consider the graphon $W_1$ defined as
\[W_1(x,y)=\begin{cases}
          0 & \mbox{if $x\in A$ or $y\in A$, and} \\
	  W(x,y) & \mbox{otherwise.}
          \end{cases}\]
Observe that
\begin{equation}
\deg_{W_1}(x)\ge p(1-0.001\varepsilon^5)-\mu(A)\geq  p(1-0.002\varepsilon^5)\label{eq:essminW1}
\end{equation}
for every $x\not\in A$.
It follows that
the density $p_1$ of $W_1$ is at least
\[ p_1 \geq (1-\mu(A)) (1-0.002p\varepsilon^5)\geq (1-0.001p\varepsilon^5)p(1-0.002\varepsilon^5) \geq p(1-0.003\varepsilon^5).\] 
Note that the cut distance between $W$ and $W_1$ is at most $2\mu(A)+\mu(A)^2\leq 3\mu(A)\leq 0.003p\varepsilon^5$.
Hence, the triangle inequality implies that
\[\cut{W_1-p_1}\ge \cut{W-p} - \cut{W-W_1} - \lvert p-p_1\rvert \ge \varepsilon-0.006p \varepsilon^5 \ge 0.994\varepsilon.\]
Lemma~\ref{lm:gammagamma} now yields that
\[t(C_4,W_1) \ge p_1^4+\frac{0.994^4\varepsilon^4}{8} 
             \ge p_1^4(1+0.122\varepsilon^4)
             \ge p^4\left(1-0.003\varepsilon^5\right)^4(1+0.122\varepsilon^4)
             \ge p^4\left(1+0.1\varepsilon^4\right).
	     \]
Hence, Lemma \ref{lm:K2n2nC4a} implies (the last inequality follows from $\ell\leq ab/4$) that
\[t(K_{a,b},W_1)\ge t(C_4, W_1)^{ab/4} 
  \geq p^{ab}\left(1+0.1\varepsilon^4\right)^{ab/4}
  \geq p^{ab}\left(1+0.03\varepsilon^4\right)^{ab/4+2\ell}.
  \]
Therefore, we obtain using \eqref{eq:essminW1} that
\[
t(K_{a|\ell,b},W_1) \ge p^{ab+\ell}\left(1+0.03\varepsilon^4\right)^{ab/4+2\ell}\left(1-0.002\varepsilon^5\right)^\ell
		    \ge p^{ab+\ell}\left(1+0.03\varepsilon^4\right)^{ab/4+\ell}.
\]		    
As $t(K_{a|\ell,b},W) \geq t(K_{a|\ell,b},W_1)$, the estimate in the statement of the lemma follows.

We next analyze the complementary case,
i.e., the case when the measure of $A$ is at least $0.001p\varepsilon^5$.
We obtain the following estimate on $t(K_{1,2},W)$ using the Cauchy-Schwarz inequality:
\begin{align*}
t(K_{1,2},W) & = 
\int_{[0,1]} \deg_W(x)^2 \dd x = \int_{A} \deg_W(x)^2 \dd x + \int_{[0,1]\setminus A} \deg_W(x)^2 \dd x \\
& \geq 
\frac{1}{\mu(A)}\left(\int_{A}\deg_W(x)\dd x\right)^2 +
 \frac{1}{1-\mu(A)}\left(\int_{[0,1]\setminus A}\deg_W(x)\dd x\right)^2 \\
& = 
\frac{1}{\mu(A)}\left(\int_{A}\deg_W(x)\dd x\right)^2 +
 \frac{1}{1-\mu(A)}\left(p-\int_{A}\deg_W(x)\dd x\right)^2 \\
& = 
 p^2 + \frac{\mu(A)}{1-\mu(A)}\left(p-\frac{\int_{A}\deg_W(x)\dd x}{\mu(A)}\right)^2 \\
& \geq 
 p^2 + \frac{0.001p\varepsilon^5}{1-0.001p\varepsilon^5}(0.001p\varepsilon^5)^2 \geq
 p^2 + 10^{-9}\varepsilon^{15}p^3.
\end{align*}
Observe that $\varepsilon=\cut{W-p}\le\cut{W}+\cut{p}=2p$ and so
$t(K_{1,2},W)\ge p^2(1+10^{-9}\varepsilon^{16}/2)$.
Therefore, Lemma~\ref{lm:entropy} now yields that
\[
t(K_{a|\ell,b},W) \ge
t(K_{1,2},W)^{(ab+\ell)/2}\ge
p^{ab+\ell}(1+10^{-9}\varepsilon^{16}/2)^{ab/2+\ell/2} \ge
p^{ab+\ell}(1+10^{-9}\varepsilon^{16})^{ab/4+\ell/4}.
\]
The estimate given in the statement of the lemma is now established.
\end{proof}

\section{Locally Sidorenko component}
\label{sec:locSidorenko}

In this section,
we establish that
a certain family of graphs that contain any fixed high girth graph as an induced subgraph
has the Sidorenko property in a local sense.
As discussed in Subsections~\ref{subsec:motivation} and~\ref{subsec:overview}, 
the main contribution of this section is Theorem~\ref{thm:core} asserting that
the value of $\varepsilon_0$ is independent of $\ell$ and so,
when $H^\bullet$ is fixed, of the graph $H^\bullet\oplus P_{\ell}^\bullet$.
The proof is inspired by the proof of an extension of a result by Lov\'asz~\cite{Lov11} due to Fox and the last author~\cite{FoxW17,FoxW},
however, it does not follow from either of these results even when the bound is not required to be uniform.
We remark that the lower bound of $50$ on the girth of $G$ given in Theorem~\ref{thm:core}
can certainly be decreased, possibly all the way to $4$,
at the expense of making the proof of Theorem~\ref{thm:core} significantly more technical;
we decided to prove the theorem with the lower bound of $50$
as proving it with an even larger lower bound on the girth would not simplify the proof substantially.

\begin{theorem}
\label{thm:core}
For every graph $G$ with girth at least $50$ and 
every real number $p_0 \in (0,1)$,
there exist a real $\varepsilon_0>0$ and
a connected rooted graph $H^\bullet$ containing $G$ as an induced subgraph
such that the following holds
for every graphon $W$ with density $p \geq p_0$ and $\cut{W-p}\leq\varepsilon_0$:
\[t(H^\bullet\oplus P_{\ell}^\bullet,W)\ge p^{\nume{H}+\ell-1}\]
for every integer $\ell\ge 9$.
\end{theorem}

\begin{proof}
Fix a graph $G$ and a real $p_0\in (0,1)$.
The graph $H^\bullet$ is obtained from the graph $G$
by attaching a $3$-edge-path and a $12$-edge-path to two different vertices of $G$,
attaching a cycle of length four to the end vertex of the $12$-edge-path not contained in $G$, and
choosing the end vertex of the $3$-edge-path not contained in $G$ as the root.
The construction of $H^\bullet$ is illustrated in Figure~\ref{fig:1}.
In what follows,
we will use $H$ for the non-rooted graph with the same vertices and edges as $H^\bullet$,
i.e., $H^\bullet$ can be obtained from $H$ be distinguishing the vertex $v_0$ as the root.

\begin{figure}
\begin{center}
\epsfbox{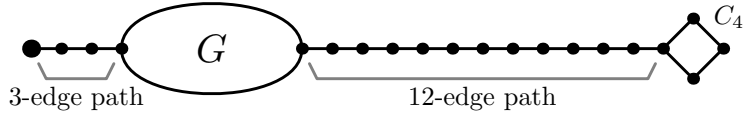}
\end{center}
\caption{The graph $H^\bullet$ from the statement of Theorem~\ref{thm:core}.}
\label{fig:1}
\end{figure}

We will show that the statement of the theorem holds for $\varepsilon_0$ chosen as
\[\varepsilon_0=\min\left\{\frac{p_0^4}{10^{12}},\frac{p_0^{64}}{2^{8\nume{H}+90}}\right\}.\]
In particular, $\varepsilon_0$ satisfies that
\begin{equation}
\frac{3\varepsilon_0^{1/4}}{p_0}\le\frac{1}{100}\left(1-\frac{4\varepsilon_0}{p_0^4}\right)^2.\label{eq:eps0}
\end{equation}
We remark that the estimate \eqref{eq:eps0} will be used later to show that
the middle inequality in \eqref{eq:Pmratio} holds for all $m\ge 2$ and
the first inequality in \eqref{eq:Plratio} holds for all $m=2,\ldots,\ell$
whenever $\lambda_1\in [p_0,1]$, $\delta\in [0,4\varepsilon_0/p_0^4]$ and $\gamma\in [0,4\varepsilon_0]$.

Consider a graph $H^\bullet\oplus P_{\ell}^\bullet$ for $\ell\ge 9$ and
a graphon $W$ with density $p\ge p_0$ such that $\cut{W-p}\leq\varepsilon_0$.
To simplify our notation, we will write $H_{\ell}$ for $H^\bullet\oplus P_{\ell}^\bullet$.
The vertices of the path $P_{\ell}^\bullet$ will be denoted by $v_0,\ldots,v_{\ell-1}$ in a way that $v_0$ is the root;
the two internal vertices of the $3$-edge path attached to $G$ when constructing $H^\bullet$
will be denoted by $v_{-2}$ and $v_{-1}$ in a way that $v_{-1}$ is the neighbor of $v_0$.
The notation is illustrated in Figure~\ref{fig:2}.

\begin{figure}
\begin{center}
\epsfbox{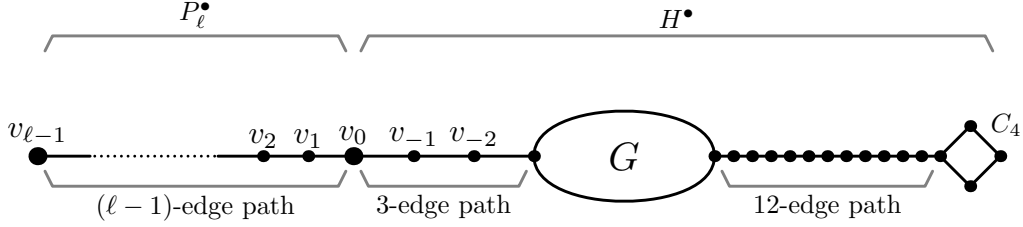}
\end{center}
\caption{The graph $H_\ell=H^\bullet\oplus P_{\ell}^\bullet$ from the statement of Theorem~\ref{thm:core} and
         notation concerning the vertices $v_{-2},v_{-1},\ldots,v_{\ell-1}$ used in its proof.}
\label{fig:2}
\end{figure}

Let $U=W-p$ and observe that $t(K_2,U)=0$ (as the density of $W$ is $p$) and $\|U\|_\infty\le 1$.
Since a path $P_{m}$ has the Sidorenko property for every $m\ge 1$,
we obtain that
\begin{equation}
t(P_{m},p+U)=t(P_{m},W)\ge p^{m-1}\label{eq:UPm}
\end{equation}
for every $m\ge 1$.
The density of $H_{\ell}$ can be expressed (see~\cite{Lov11} and~\cite[Proof of Proposition 16.27]{Lov12}) as
\begin{equation}
t(H_{\ell},W)=t(H_{\ell},p+U)=
p^{\nume{H}+\ell-1}+\sum_{\substack{F_0\subseteq E(H_{\ell})\\F_0\not=\emptyset}}p^{\nume{H}+\ell-1-\lvert F_0\rvert}t(\langle F_0\rangle,U).
\label{eq:goal0}
\end{equation}
Hence, 
in order to establish the statement of the theorem,
it is enough to show the following (note that the sum from~\eqref{eq:goal0} is divided by $p^{\nume{H}}$):
\begin{equation}
\sum_{\substack{F_0\subseteq E(H_{\ell})\\F_0\not=\emptyset}}p^{\ell-1-\lvert F_0\rvert}t(\langle F_0\rangle,U)\ge 0.
\label{eq:goal}
\end{equation}
We split the non-empty $2^{\nume{H}+\ell-1}-1$ subsets $F_0$ of $E(H_{\ell})$ into eight groups given below,
which we analyze separately.
We say that a component of $\langle F_0\rangle$ is a \emph{core} component if it contains an edge of $H^\bullet$.
\begin{enumerate}[label=(\alph*)]
\item The graph $\langle F_0\rangle$ has no core component. \label{it:A}
\item The graph $\langle F_0\rangle$ has a single core component and this component is $C_4$. \label{it:B}
\item The graph $\langle F_0\rangle$ has a single core component and this component is a star. \label{it:C}
\item The graph $\langle F_0\rangle$ has a core component containing the vertex $v_8$. \label{it:D}
\item The graph $\langle F_0\rangle$ has no core component containing $v_8$ and there is a core component containing a cycle of $G$. \label{it:E}
\item The graph $\langle F_0\rangle$ has no core component containing $v_8$ and the only cycle in $\langle F_0\rangle$ is $C_4$ but the core component containing $C_4$ is not just $C_4$. \label{it:F}
\item The graph $\langle F_0\rangle$ has no core component containing $v_8$, all core components are acyclic or $C_4$, and at least one of the acyclic core components is not a star. \label{it:G}
\item The graph $\langle F_0\rangle$ has no core component containing $v_8$, all core components are stars or $C_4$, and there are at least two core components. \label{it:H}
\end{enumerate}
We briefly verify that each subset $F_0$ belongs to exactly one of the groups \ref{it:A}--\ref{it:H}.
Since any core component containing the vertex $v_8$ can be neither $C_4$ nor a star,
the group \ref{it:D} is disjoint from the remaining seven groups.
If $F_0$ is not covered by the group \ref{it:A} or \ref{it:D},
then $\langle F_0\rangle$ has at least one core component and no core component of $\langle F_0\rangle$ contains $v_8$.
We restrict our attention to such subsets $F_0$ in the rest of this paragraph.
If $\langle F_0\rangle$ has a core component containing a cycle of $G$ (note that the length of this cycle is at least $50$),
then $F_0$ belongs to the group \ref{it:E} and it cannot belong to any other group.
If $\langle F_0\rangle$ has no core component containing a cycle of $G$
but it has a core component that contains the cycle $C_4$ and this core component is not solely this cycle,
then $F_0$ belongs to the group \ref{it:F} and it cannot belong to any other group.
Hence, we can further restrict our attention to only those subsets $F_0$ such that
no core component of $\langle F_0\rangle$ contains $v_8$ and
each core component of $\langle F_0\rangle$ is either acyclic or solely the cycle $C_4$.
If one of core components is not a star, then $F_0$ belongs to the group \ref{it:G} and it cannot belong to any other group.
Otherwise, all core components of $\langle F_0\rangle$ are stars or solely the cycle $C_4$.
If $F_0$ has a single core component, then it belongs to the group \ref{it:B} or the group \ref{it:C};
otherwise, i.e., when $F_0$ has at least two core components, it belongs the group \ref{it:H}.
We conclude that each subset $F_0$ belongs to exactly one of the groups \ref{it:A}--\ref{it:H}.

We next analyze the contribution of the terms for $F_0$
based on the groups \ref{it:A}--\ref{it:H}
in Claims~\ref{cl:ABC}--\ref{cl:DEFGH}.
Claim~\ref{cl:ABC} and Claim~\ref{cl:DEFGH} will readily yield that
the left hand side of \eqref{eq:goal} is at least
\begin{align*}
  & t(P_{\ell},p+U)(t(P_3,U)+t(C_4,U))-t(P_{\ell},p+U)\left(\frac{t(P_3,U)}{2}+\frac{3t(C_4,U)}{8}\right)\\
  \ge\; & t(P_{\ell},p+U)\left(\frac{t(P_3,U)}{2}+\frac{5t(C_4,U)}{8}\right)\ge 0,
\end{align*}
which establishes that \eqref{eq:goal} holds.
In particular, when Claim~\ref{cl:ABC} and Claim~\ref{cl:DEFGH} are established,
the proof of the theorem will be completed.
We remark that Claims \ref{cl:D}--\ref{cl:H} describe auxiliary results used to prove Claim~\ref{cl:DEFGH}.

\begin{claim}
\label{cl:ABC}
The sum of the terms in the left side of \eqref{eq:goal} that
involve $F_0$ from one of the groups \ref{it:A}--\ref{it:C}
is at least
\[t(P_{\ell},p+U)(t(P_3,U)+t(C_4,U)).\]
\end{claim}

\begin{proof}[Proof of Claim~\ref{cl:ABC}]
The sum of the terms that involve a set $F_0$ from the group~\ref{it:A} is equal to
\begin{equation}
\sum_{\substack{F\subseteq E(P_{\ell})\\F\not=\emptyset}}p^{\ell-1-\lvert F\rvert}t(\langle F\rangle,U)=
\sum_{F\subseteq E(P_{\ell})}p^{\ell-1-\lvert F\rvert}t(\langle F\rangle,U)-p^{\ell-1}=t(P_{\ell},p+U)-p^{\ell-1}\ge 0,
\label{eq:grA}
\end{equation}
where the last inequality follows from the fact that $P_{\ell}$ has the Sidorenko property and $p+U=W$.

Along the same lines,
we derive that the sum of the terms that involve a set $F_0$ from the group \ref{it:B} 
(note that such $F_0$ contain all the edges of $C_4$ and some edges of $P_{\ell}^\bullet$, and
$F_0$ can consist of the four edges of $C_4$ only)
is equal to
\begin{equation}
\sum_{F\subseteq E(P_{\ell})}p^{\ell-1-\lvert F\rvert-4}t(\langle F\rangle,U)t(C_4,U)
  =t(P_{\ell},p+U)\frac{t(C_4,U)}{p^4}
  \ge t(P_{\ell},p+U)t(C_4,U).
\label{eq:grB}
\end{equation}

It remains to analyze the sum of terms that involve a set $F_0$ from the group \ref{it:C}.
For any edge $e$ of $H^\bullet$,
the sum of the terms that involve a set $F_0$ from the group \ref{it:C} such that
the core component of $\langle F_0\rangle$ is formed by the edge $e$ only
is equal to zero as $t(K_2,U)=0$.
Hence, the sum of the terms that involve a set $F_0$ from the group \ref{it:C}
is equal to
\begin{itemize}
\item the sum of the terms that involve such $F_0$ with $\{v_{-2}v_{-1},v_{-1}v_0\}\subseteq F_0$,
\item the sum of the terms that involve such $F_0$ with $\{v_{-1}v_0,v_0v_1\}\subseteq F_0$, and
\item the sum of the terms that involve such $F_0$ that the core component of $F_0$ has at least two edges but does not contain $v_0$.
\end{itemize}
The sum of the terms that involve such $F_0$ with $\{v_{-2}v_{-1},v_{-1}v_0\}\subseteq F_0$ is equal to
(note that $t(K_{1,2},U)\ge 0$ by Lemma~\ref{lm:L1})
\begin{equation}
\sum_{F\subseteq E(P_{\ell-1})}p^{\ell-2-\lvert F\rvert-1}t(\langle F\rangle,U)t(K_{1,2},U)=
t(P_{\ell-1},p+U)\frac{t(K_{1,2},U)}{p}\ge 0.\label{eq:grC1}
\end{equation}
Similarly, the sum of the terms that involve such $F_0$ with $\{v_{-1}v_0,v_0v_1\}\subseteq F_0$ is equal to
\begin{equation}
\sum_{F\subseteq E(P_{\ell-2})}p^{\ell-3-\lvert F\rvert}t(\langle F\rangle,U)t(K_{1,2},U)=
t(P_{\ell-2},p+U)t(K_{1,2},U)\ge 0.\label{eq:grC2}
\end{equation}
Finally,
sum of the terms that involve such $F_0$ from the group~\ref{it:C} such that
the core component of $F_0$ has at least two edges but does not contain $v_0$
is equal to the following:
\begin{align}
& \sum_{F\subseteq E(P_{\ell})}\sum_{v\in V(H)\setminus\{v_0,v_{-1}\}}\sum_{d=2}^{\deg_H(v)}\binom{\deg_H(v)}{d}p^{\ell-1-\lvert F\rvert-d}t(\langle F\rangle,U)t(K_{1,d},U) \nonumber \\
= & \left(\sum_{F\subseteq E(P_{\ell})}p^{\ell-1-\lvert F\rvert}t(\langle F\rangle,U)\right)\sum_{v\in V(H)\setminus\{v_0,v_{-1}\}}\sum_{d=2}^{\deg_H(v)}\binom{\deg_H(v)}{d}p^{-d}t(K_{1,d},U) \nonumber \\
= & t(P_{\ell},p+U)\sum_{v\in V(H)\setminus\{v_0,v_{-1}\}}\sum_{d=2}^{\deg_H(v)}\binom{\deg_H(v)}{d}p^{-d}t(K_{1,d},U). \label{eq:grC3z}
\end{align}
Since it holds that $t(K_{1,1},U)=t(K_2,U)=0$, we obtain that \eqref{eq:grC3z} is equal to
\begin{align}
  & t(P_{\ell},p+U)\sum_{v\in V(H)\setminus\{v_0,v_{-1}\}}\sum_{d=1}^{\deg_H(v)}\binom{\deg_H(v)}{d}p^{-d}t(K_{1,d},U) \nonumber \\
= & t(P_{\ell},p+U)\sum_{v\in V(H)\setminus\{v_0,v_{-1}\}}\sum_{d=1}^{\deg_H(v)}\binom{\deg_H(v)}{d}\int\limits_{[0,1]}\left(\frac{t_x(K_2^\bullet,U)}{p}\right)^d\dd x \nonumber \\
= & t(P_{\ell},p+U)\sum_{v\in V(H)\setminus\{v_0,v_{-1}\}}\int\limits_{[0,1]}\left(1+\frac{t_x(K_2^\bullet,U)}{p}\right)^{\deg_H v}-1\dd x. \label{eq:grC3x}
\end{align}
Recall that Pascal's inequality asserts $(1+x)^n\ge 1+nx$ for all integers $n\ge 0$ and reals $x\ge -1$.
As $U\ge -p$, it holds that $t_x(K_2^\bullet,U)\ge -p\ge -1$ and so \eqref{eq:grC3x} is at least
\begin{align}
& t(P_{\ell},p+U)\sum_{v\in V(H)\setminus\{v_0,v_{-1}\}}\int\limits_{[0,1]}\left(1+\frac{(\deg_H v-1)t_x(K_2^\bullet,U)}{p}\right)\left(1+\frac{t_x(K_2^\bullet,U)}{p}\right)-1\dd x \nonumber \\
= & t(P_{\ell},p+U)\sum_{v\in V(H)\setminus\{v_0,v_{-1}\}}\int\limits_{[0,1]}\frac{\deg_H v\cdot t_x(K_2^\bullet,U)}{p}+\frac{(\deg_H v-1)t_x(K_2^\bullet,U)^2}{p^2}\dd x \nonumber \\
= & t(P_{\ell},p+U)\sum_{v\in V(H)\setminus\{v_0,v_{-1}\}}\deg_H v\;\frac{t(K_2,U)}{p}+(\deg_H v-1)\frac{t(K_{1,2},U)}{p^2} \nonumber \\
= & t(P_{\ell},p+U)\sum_{v\in V(H)\setminus\{v_0,v_{-1}\}}(\deg_H v-1)\frac{t(K_{1,2},U)}{p^2}
\ge t(P_{\ell},p+U)t(K_{1,2},U),
\label{eq:grC3}
\end{align}
where the last inequality follows
since $p\le 1$ and the degree of at least one vertex of $H$ different from $v_0$ and $v_{-1}$ is at least two.

Since the quantities estimated in \eqref{eq:grA}, \eqref{eq:grC1} and \eqref{eq:grC2} are non-negative,
the sum of the terms that involve a set $F_0$ from one of the groups~\ref{it:A}--\ref{it:C}
is at least the sum of the final estimates given in \eqref{eq:grB} and \eqref{eq:grC3},
i.e., at least $t(P_{\ell},p+U)(t(K_{1,2},U)+t(C_4,U))$.
This yields the bound stated in the claim.
\end{proof}

Before we proceed with proving Claims~\ref{cl:D}--\ref{cl:H},
we recall an estimate on the density of $P_{\ell}$ in $W=p+U$ from Subsection~\ref{subsec:spectral} and
provide more refined estimates that are needed in the proofs of some of the claims.
Let $\lambda_1$ be the largest eigenvalue of $W$ viewed as a Hilbert-Schmidt operator on $L_2[0,1]$,
let $\gamma=t(C_4,W)-p^4$ and
let $\delta$ be defined as in Subsection~\ref{subsec:spectral}.
Recall that $\gamma\le 4\varepsilon_0$ by \eqref{eq:cutU},
$p\le\lambda_1\le p+\varepsilon_0/p^3$ by \eqref{eq:lambda1small}, and
$\delta\le 4\varepsilon_0/p^4$ by \eqref{eq:delta}.
By \eqref{eq:Pmrange}, it holds that
\[
\lambda_1^{m}(1-\delta)^2-\gamma^{m/4}\le t(P_{m+1},p+U)\le\lambda_1^{m}(1-\delta)^2+\gamma^{m/4}
\]
for every integer $m\ge 0$.
Next, the choice of $\varepsilon_0$ implies that
$\gamma^{1/4}\le (4p_0^4/10^{12})^{1/4}\le p_0/700$ and
$\delta\le 4\varepsilon_0/p^4\le 10^{-11}$, and
so we obtain that
\begin{align}
\frac{t(P_m,p+U)}{t(P_{m+1},p+U)}
& \le\frac{\lambda_1^{m-1}(1-\delta)^2+\gamma^{(m-1)/4}}{\lambda_1^{m}(1-\delta)^2-\gamma^{m/4}} \nonumber \\
& \le\frac{\lambda_1^{m-1}(1-10^{-11})^2+(p_0/700)^{m-1}}{\lambda_1^{m}(1-10^{-11})^2-(p_0/700)^m} \nonumber \\
& \le\frac{\lambda_1^{m-1}(1-10^{-11})^2+(\lambda_1/700)^{m-1}}{\lambda_1^{m}(1-10^{-11})^2-(\lambda_1/700)^m} \nonumber \\
& \le\frac{(1-10^{-11})^2+1/700}{(1-10^{-11})^2-1/700^2}\cdot\frac{1}{\lambda_1}
  \le \frac{101}{100\lambda_1} \le \frac{101}{100p}
\label{eq:Pmratio}
\end{align}
for every integer $m\ge 2$;
since $t(P_1,p+U)=1$ and $t(P_2,p+U)=p$,
the estimate~\eqref{eq:Pmratio} also holds for $m=1$.
Hence,
the following holds for every $i=0,\ldots,8$:
\begin{equation}
t(P_{\ell-i},p+U)\le \left(\frac{101}{100p}\right)^i t(P_{\ell},p+U)
\le \frac{2}{p^8}t(P_{\ell},p+U)
\le \frac{2}{p_0^{8}}t(P_{\ell},p+U).
\label{eq:Pmsum}
\end{equation}
Similarly to \eqref{eq:Pmratio},
the choice of $\varepsilon_0$ implies that the following holds for every integer $m=2,\ldots,\ell$:
\begin{equation}
\frac{t(P_{m},p+U)}{t(P_{\ell},p+U)^{(m-1)/(\ell-1)}}
\le\frac{\lambda_1^{m-1}(1-\delta)^2+\gamma^{(m-1)/4}}{\left(\lambda_1^{\ell-1}(1-\delta)^2-\gamma^{(\ell-1)/4}\right)^{\frac{m-1}{\ell-1}}}
\le\frac{101}{100}.
\label{eq:Plratio}
\end{equation}
If $m=1$, the left side of \eqref{eq:Plratio} is $1$, and
so the estimate~\eqref{eq:Plratio} also holds for $m=1$.

\begin{claim}
\label{cl:D}
The absolute value of the sum of the terms in the left side of \eqref{eq:goal} that
involve $F_0$ from the group \ref{it:D}
is at most
\[\frac{1}{8}t(P_{\ell},p+U)t(C_4,U).\]
\end{claim}

\begin{proof}[Proof of Claim~\ref{cl:D}]
Consider a set $F_0$ from the group \ref{it:D}.
Let $F_c$ be the edges of $F_0$ contained in the core components of $\langle F_0\rangle$, and
let $k$ be the largest index such that $v_k$ is contained in a core component of $F_0$.
Note $k\ge 8$ by the definition of the group~\ref{it:D}.
Let $S^\bullet$ be the rooted graph obtained from $\langle F'\rangle$ by removing the path $v_0\cdots v_k$ and
choosing the vertex $v_0$ to be the root.
Observe that
\begin{align*}
\left\lvert t(\langle F_c\rangle,U)\right\rvert
& = \left\lvert\int_{[0,1]}t_x(P_{k+1}^\bullet,U)t_x(S^\bullet,U)\dd x\right\rvert \\
& \le \left(\int_{[0,1]}t_x(P_{k+1}^\bullet,U)^2\dd x\right)^{1/2}\left(\int_{[0,1]}t_x(S^\bullet,U)^2\dd x\right)^{1/2} \\
& \le \left(\int_{[0,1]}t_x(P_{k+1}^\bullet,U)^2\dd x\right)^{1/2}\dd x=t(P_{2k+1},U)^{1/2}\le t(C_4,U)^{k/4},
\end{align*}
where the last inequality follows from Lemma~\ref{lm:L1b}.
As there are at most $2^{\nume{H}}$ choices of a set $F_c$ such that a core component of $F_0$ contains $v_k$ but not $v_{k+1}$,
the absolute value of the sum of the terms that involve $F_0$ from the group \ref{it:D} is at most
\begin{equation}
2^{\nume{H}}\left(t(C_4,U)^{\ell/4}+\sum_{k=8}^{\ell-1}t(P_{\ell-k},p+U)t(C_4,U)^{k/4}\right).
\label{eq:grDsum}
\end{equation}
Using \eqref{eq:Plratio}, we estimate the sum in \eqref{eq:grDsum} as follows:
\begin{align*}
t(C_4,U)^{\ell/4}+\sum_{k=8}^{\ell-1}t(P_{\ell-k},p+U)t(C_4,U)^{k/4}
& \le t(C_4,U)^{\ell/4}+\frac{101}{100}\sum_{k=8}^{\ell-1}t(P_{\ell},p+U)^{\frac{\ell-k-1}{\ell-1}}t(C_4,U)^{k/4} \\
& \le \frac{101}{100}\sum_{k=8}^{\ell}t(P_{\ell},p+U)^{\frac{\ell-k-1}{\ell-1}}t(C_4,U)^{k/4}.
\end{align*}
We next estimate the last sum using that $t(P_{\ell},p+U)\ge p^{\ell-1}$:
\begin{align*}
\frac{101}{100} \sum_{k=8}^{\ell}t(P_{\ell},p+U)^{\frac{\ell-k-1}{\ell-1}}t(C_4,U)^{k/4}
& \le \frac{101}{100}t(P_{\ell},p+U)\sum_{k=8}^{\ell}\left(\frac{t(C_4,U)^{1/4}}{t(P_{\ell},p+U)^{\frac{1}{\ell-1}}}\right)^k \\
& \le \frac{101}{100}t(P_{\ell},p+U)\sum_{k=8}^{\ell}\left(\frac{t(C_4,U)^{1/4}}{p}\right)^k \\
& = \frac{101}{100}t(P_{\ell},p+U)\frac{t(C_4,U)^2}{p^8}\sum_{m=0}^{\ell-8}\left(\frac{t(C_4,U)^{1/4}}{p}\right)^m \\
& \le \frac{101}{100}t(P_{\ell},p+U)\frac{t(C_4,U)^2}{p^8}\sum_{m=0}^{\ell-8}\left(\frac{(4\varepsilon_0)^{1/4}}{p}\right)^m,
\end{align*}
which we further estimate using $\varepsilon_0\le p_0^4/10^{12}$ to be at most
\begin{align*}
\frac{101}{100}t(P_{\ell},p+U)\frac{t(C_4,U)^2}{p^8}\sum_{m=0}^{\ell-8}\left(\frac{p_0}{100p}\right)^m
& \le \frac{101}{100}t(P_{\ell},p+U)\frac{t(C_4,U)^2}{p^8}\sum_{m=0}^{\infty}\left(\frac{1}{100}\right)^m\\
& \le 2t(P_{\ell},p+U)\frac{t(C_4,U)^2}{p^8}.
\end{align*}
We now plug in this estimate to \eqref{eq:grDsum} and
derive using \eqref{eq:cutU} and the choice of $\varepsilon_0$ that
the absolute value of the sum of the terms that involve $F_0$ from the group \ref{it:D} is at most
\[
2^{\nume{H}+1}\cdot\frac{t(P_{\ell},p+U)t(C_4,U)^2}{p^8} 
\le 2^{\nume{H}+1}\cdot\frac{4\varepsilon_0 t(P_{\ell},p+U)t(C_4,U)}{p^8}
\le\frac{1}{8}t(P_{\ell},p+U)t(C_4,U).\]
The proof of the claim is now completed.
\end{proof}

\begin{claim}
\label{cl:E}
The absolute value of the sum of the terms in the left side of \eqref{eq:goal} that
involve $F_0$ from the group \ref{it:E}
is at most
\[\frac{1}{8}t(P_{\ell},p+U)\left(t(P_3,U)+t(C_4,U)\right).\]
\end{claim}

\begin{proof}[Proof of Claim~\ref{cl:E}]
Consider a set $F_0$ such that a core component of $\langle F_0\rangle$ contains a cycle of $G$.
Let $F_c$ be the edges of $F_0$ contained in the core components of $\langle F_0\rangle$,
let $C$ be the shortest cycle of $G$ contained in $\langle F_c\rangle$,
let $L\ge 50$ be the length of $C$, and
let $q$ be the number of vertices of degree at least three in $\langle F_c\rangle$ that are contained in $C$.
If $q\le 7$, a part of the cycle $C$ is an $8$-edge path whose all internal vertices have degree two in $\langle F_c\rangle$.
Let $S^{\bullet\bullet}$ be the rooted graph obtained from $\langle F_c\rangle$ by removing the edges of such an $8$-edge path and
choosing the end vertices of the removed path to be two roots of $S^{\bullet\bullet}$.
We obtain using the Cauchy-Schwarz inequality and Lemma~\ref{lm:L1a} that
\begin{align}
\left\lvert t(\langle F_c\rangle,U)\right\rvert
& = \left\lvert\int_{[0,1]^2}t_{x,x'}(P_9^{\bullet\bullet},U)t_{x,x'}(S^{\bullet\bullet},U)\dd x\dd x'\right \rvert \nonumber \\
& \le \left(\int_{[0,1]}t_{x,x'}(P_9^{\bullet\bullet},U)^2\dd x\dd x'\right)^{1/2}
      \left(\int_{[0,1]}t_{x,x'}(S^{\bullet\bullet},U)^2\dd x\dd x'\right)^{1/2} \nonumber \\
& \le t(C_{16},U)^{1/2}\le t(C_4,U)^2. \label{eq:grE1}
\end{align}

\begin{figure}
\begin{center}
\epsfbox{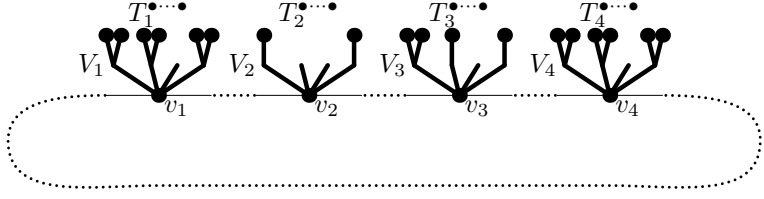}
\end{center}
\caption{Illustration of the definition of $v_i$, $V_i$ and $T_i^{\bullet\cdots\bullet}$
         in the proof of Claim~\ref{cl:E};
	 roots are the vertices depicted in the figure.}
\label{fig:grET}
\end{figure}

\begin{figure}
\begin{center}
\epsfbox{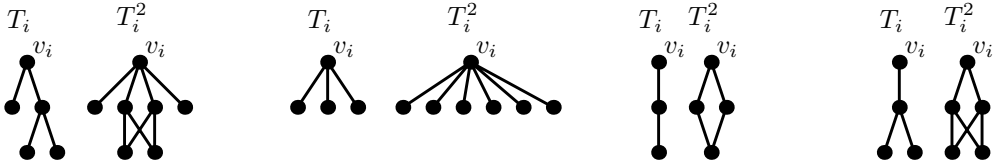}
\end{center}
\caption{Examples of rooted trees $T_i^{\bullet\cdots\bullet}$ and the associated trees $T_i^2$
         in the proof of Claim~\ref{cl:E}.}
\label{fig:grE}
\end{figure}

We next assume that $q\ge 8$.
Let $v_1,\ldots,v_4$ be any four vertices of $C$ with degree at least three in $\langle F_c\rangle$ such that
no two of them are consecutive on the cycle $C$;
the choice of $C$ as the shortest cycle contained in $G$ implies that the vertices $v_1,\ldots,v_4$ form an independent set.
Let $V_i$ for $i\in [4]$ be the set consisting of the vertex $v_i$ and
the vertices at distance at most two from $v_i$ in $\langle F_c\rangle$ that are not contained in the cycle $C$;
see Figure~\ref{fig:grET}.
Observe that the sets $V_1,\ldots,V_4$ are disjoint and there is no edge joining a vertex of $V_i$ and $V_j$ for $i\not=j$;
otherwise, $\langle F_c\rangle$ would contain a cycle shorter than $C$,
which would be formed by a path from $v_i$ through $V_i$ and then $V_j$ back to $v_j$ and
the shorter of the two parts of $C$ delimited by $v_i$ and $v_j$.
For $i=1,\ldots,4$,
let $T_i^{\bullet\cdots\bullet}$ be the rooted graph obtained from the subgraph of $\langle F_c\rangle$ induced by $V_i$
by choosing the vertex $v_i$ and all vertices at distance exactly two from $v_i$ to be the roots.
Further, let $S^{\bullet\cdots\bullet}$ be the graph obtained from $\langle F_c\rangle$
by deleting the non-root vertices of $T_1^{\bullet\cdots\bullet},\ldots,T_4^{\bullet\cdots\bullet}$ and
choosing the other vertices shared with $T_1^{\bullet\cdots\bullet},\ldots,T_4^{\bullet\cdots\bullet}$ as roots.
Let $m$ be the number of roots of $S^{\bullet\cdots\bullet}$.
Observe that it holds that (we omit the exact subscripts at the functions $t$ for clarity)
\[\left\lvert t(\langle F_c\rangle,U)\right\rvert=
  \left\lvert\int_{[0,1]^m}t_{\cdots}(S^{\bullet\cdots\bullet},U)\prod_{i=1}^4 t_{\cdots}(T_i^{\bullet\cdots\bullet},U)\dd x_1\cdots\dd x_m\right\rvert.\]
For $i=1,\ldots,4$, we write $T_i^2$ for the graph obtained by taking two copies of $T_i^{\bullet\cdots\bullet}$ and
identifying the pairs of corresponding roots; see Figure~\ref{fig:grE}.
By Lemma~\ref{lm:CS}, we obtain using that $\|U\|_{\infty}\le 1$ that
\begin{align}
\left\lvert t(\langle F_c\rangle,U)\right\rvert
& \le \left(\int_{[0,1]^m}t_{\cdots}(S^{\bullet\cdots\bullet},U)^2\dd x_1\cdots\dd x_m\right)^{1/2}
  \prod_{i=1}^4\left(\int_{[0,1]^m}t_{\cdots}(T_i^{\bullet\cdots\bullet},U)^2\dd x_1\cdots\dd x_m\right)^{1/2}\nonumber\\
& \le\prod_{i=1}^4\left(\int_{[0,1]^m}t_{\cdots}(T_i^{\bullet\cdots\bullet},U)^2\dd x_1\cdots\dd x_m\right)^{1/2}=
  \prod_{i=1}^4 t(T_i^2,U)^{1/2}.
\label{eq:grECS}
\end{align}
Consider $i\in [4]$.
If the vertex $v_i$ has degree one in $T_i^{\bullet\cdots\bullet}$,
then $T_i^2$ is the complete bipartite graph $K_{2,d}$ where $d$ is the number of roots of $T_i^{\bullet\cdots\bullet}$.
If $d=1$, then $t(T_i^2,U)=t(P_3,U)$.
If $d\ge 2$, we obtain using Lemma~\ref{lm:L2} that $t(T_i^2,U)\le t(C_4,U)$.
If the vertex $v_i$ has a neighbor of degree one in $T_i^{\bullet\cdots\bullet}$,
it holds that $t_{x_1,\ldots,x_r}(T_i^{\bullet\cdots\bullet},U)^2\le t_{x_1}(P_2^\bullet,U)^2$ (assuming that $x_1$
corresponds to the root vertex $v_i$) and
so $t(T_i^2,U)\le t(P_3,U)$.
Otherwise, $T_i^2$ satisfies the assumptions of Lemma~\ref{lm:L3a} and
we obtain that
\[t(T_i^2,U)\le t(C_4,U)^{5/4}\le t(C_4,U).\]
In all three cases, it holds that
\[t(T_i^2,U)\le\max\{t(P_3,U),t(C_4,U)\}=t(P_3,U)+t(C_4,U).\]
Hence, we obtain from \eqref{eq:grECS} that
\begin{equation}
\left\lvert t(\langle F_c\rangle,U)\right\rvert\le\left(t(P_3,U)+t(C_4,U)\right)^2.\label{eq:grE2}
\end{equation}
As there are at most $2^{\nume{H}+7}$ possible sets $F_c$ such that an associated set $F_0$ can be contained in the group~\ref{it:E},
we conclude using \eqref{eq:Pmsum}, \eqref{eq:grE1} and \eqref{eq:grE2} that
the absolute value of the sum of the terms that involve $F_0$ from the group \ref{it:E} is at most
\begin{equation}
\frac{2}{p_0^{8}}t(P_{\ell},p+U)2^{\nume{H}+7}\left(t(P_3,U)+t(C_4,U)\right)^2.
\label{eq:grEalmost}
\end{equation}
The choice of $\varepsilon_0$ implies using \eqref{eq:cutU} that
\begin{equation}
t(C_4,U)\le 4p_0^{64}/2^{8e(H)+90}
\label{eq:C4eps0}
\end{equation}
and so
\begin{equation}
t(P_3,U)\le t(C_4,U)^{1/2}\le 2p_0^{32}/2^{4e(H)+45},
\label{eq:P3eps0}
\end{equation}
which yields that \eqref{eq:grEalmost} does not exceed $\frac{1}{8}t(P_{\ell},p+U)\left(t(P_3,U)+t(C_4,U)\right)$.
\end{proof}

\begin{claim}
\label{cl:F}
The absolute value of the sum of the terms in the left side of \eqref{eq:goal} that
involve $F_0$ from the group \ref{it:F}
is at most
\[\frac{1}{8}t(P_{\ell},p+U)\left(t(P_3,U)+t(C_4,U)\right).\]
\end{claim}

\begin{proof}[Proof of Claim~\ref{cl:F}]
Consider a set $F_0$ from the group \ref{it:F} and
let $F_c$ be the edges of $F_0$ contained in the core components of $\langle F_0\rangle$.
Let $m$ be the number of edges of the $12$-edge path from $C_4$ to the graph $G$ in $H$ that
are contained in the component of $\langle F_0\rangle$ that contains $C_4$.
Note that $m\ge 1$ by the definition of the group~\ref{it:F}.
If $m\le 7$, we obtain using $\|U\|_{\infty}\le 1$ and Lemma~\ref{lm:L3b} that
\begin{equation}
t(\langle F_c\rangle,U)\le\frac{1}{2}\left(t(C_4,U)+t(P_3,U)\right)t(C_4,U)^{1/8}.
\label{eq:grF1}
\end{equation}
If $m\ge 8$, 
let $S^{\bullet\bullet}$ be the rooted graph obtained from $\langle F_c\rangle$ 
by removing the edges of an $8$-edge subpath of the $m$-edge path starting at $C_4$ and
choosing the end vertices of the path to be two roots of $S^{\bullet\bullet}$.
As in \eqref{eq:grE1}, we obtain that
\begin{equation}
\left\lvert t(\langle F_c\rangle,U)\right\rvert\le t(C_4,U)^2. \label{eq:grF2}
\end{equation}
The estimates \eqref{eq:grF1} and \eqref{eq:grF2} imply that
the following holds in either of the cases:
\[\left\lvert t(\langle F_c\rangle,U)\right\rvert\le\left(t(C_4,U)+t(P_3,U)\right)t(C_4,U)^{1/8}.\]
As there are at most $2^{\nume{H}+7}$ possible sets $F_c$ such that an associated $F_0$ can be contained in the group~\ref{it:F},
we conclude using \eqref{eq:Pmsum} and the choice of $\varepsilon_0$,
which implies that \eqref{eq:C4eps0} holds, that
the absolute value of the sum of the terms that involve $F_0$ from the group \ref{it:F} is at most
\[\frac{2}{p_0^{8}}t(P_{\ell},p+U)2^{\nume{H}+7}\left(t(P_3,U)+t(C_4,U)\right)t(C_4,U)^{1/8}
 \le \frac{1}{8}t(P_{\ell},p+U)\left(t(P_3,U)+t(C_4,U)\right).\]
This completes the proof of the claim.
\end{proof}

\begin{claim}
\label{cl:G}
The absolute value of the sum of the terms in the left side of \eqref{eq:goal} that
involve $F_0$ from the group \ref{it:G}
is at most
\[\frac{1}{8}t(P_{\ell},p+U)t(P_3,U).\]
\end{claim}

\begin{proof}[Proof of Claim~\ref{cl:G}]
Consider a set $F_0$ from the group \ref{it:G} and
let $F_c$ be the edges of $F_0$ contained in the core components of $\langle F_0\rangle$.
We obtain using Lemma~\ref{lm:L3c} that
\[\left\lvert t(\langle F_c\rangle,U)\right\rvert\le t(P_3,U)t(C_4,U)^{1/4}.\]
As there are at most $2^{\nume{H}+7}$ possible sets $F_c$ such that an associated set $F_0$ can be contained in the group~\ref{it:G},
we conclude using \eqref{eq:Pmsum} and the choice of $\varepsilon_0$,
which implies that \eqref{eq:C4eps0} holds, that
the absolute value of the sum of the terms that involve $F_0$ from the group \ref{it:G} is at most
\[
\frac{2}{p_0^{8}}t(P_{\ell},p+U)2^{\nume{H}+7}t(P_3,U)t(C_4,U)^{1/4}\le
\frac{1}{8}t(P_{\ell},p+U)t(P_3,U).\]
The proof of the claim is now completed.
\end{proof}

\begin{claim}
\label{cl:H}
The absolute value of the sum of the terms in the left side of \eqref{eq:goal} that
involve $F_0$ from the group \ref{it:H}
is at most
\[\frac{1}{8}t(P_{\ell},p+U)t(P_3,U).\]
\end{claim}

\begin{proof}[Proof of Claim~\ref{cl:H}]
Consider a set $F_0$ from the group \ref{it:H} and
let $F_c$ be the edges of $F_0$ contained in the core components of $\langle F_0\rangle$, and
let $S_1$ and $S_2$ be any two core components of $\langle F\rangle$.
If $S_i$ is a single edge, then $t(S_i,U)=0$, and
if $S_i$ is $C_4$, then $t(S_i,U)=t(C_4,U)$.
Finally,
if $S_i$ is a star with at least two leaves, we obtain using Lemma~\ref{lm:L2} that
\[\lvert t(S_i,U)\rvert\le t(K_1,U)=t(P_3,U).\]
Hence, we obtain that
\[\left\lvert t(\langle F_c\rangle,U)\right\rvert\le \lvert t(S_1,U)t(S_2,U)\rvert\le t(P_3,U)\left(t(P_3,U)+t(C_4,U)\right).\]
As there are at most $2^{\nume{H}+7}$ possible sets $F_c$ such that an associated set $F_0$ can be contained in the group~\ref{it:H},
we conclude using \eqref{eq:Pmsum} and the choice of $\varepsilon_0$,
which implies that \eqref{eq:P3eps0} holds, that
the absolute value of the sum of the terms that involve $F_0$ from the group \ref{it:H} is at most
\[
\frac{2}{p_0^{8}}t(P_{\ell},p+U)2^{\nume{H}+7}t(P_3,U)\left(t(P_3,U)+t(C_4,U)\right)\le
\frac{1}{8}t(P_{\ell},p+U)t(P_3,U),\]
which completes the proof of the claim.
\end{proof}

Summing the bounds obtained in Claims~\ref{cl:D}--\ref{cl:H} yields the following.

\begin{claim}
\label{cl:DEFGH}
The absolute value of the sum of the terms in the left side of \eqref{eq:goal} that
involve a set $F_0$ from one of the groups \ref{it:D}--\ref{it:H}
is at most
\[t(P_{\ell},p+U)\left(\frac{t(P_3,U)}{2}+\frac{3t(C_4,U)}{8}\right).\]
\end{claim}

Since we have proven both Claim~\ref{cl:ABC} and Claim~\ref{cl:DEFGH},
the proof of the theorem is now completed.
\end{proof}

\section{Local regime}
\label{sec:local}

In this section,
we prove that the graph $H^\bullet\oplus K_{m|\ell,n}$ is locally Sidorenko (assuming $W$ has density bounded away from zero),
where the bound on locality does not depend on the parameters $\ell$, $m$ and $n$,
under the assumption that all three parameters are sufficiently large.
Similarly to the proof of Theorem~\ref{thm:core},
the proof of Theorem~\ref{thm:universal} is split into several separate claims,
which are stated and proven inside the proof of the theorem
since they rely on the notation introduced at the beginning of the proof of the theorem.

\begin{theorem}
\label{thm:universal}
For every graph $G$ with girth at least $50$ and 
every real number $p_0 \in (0,1)$,
there exist a real $\gamma_0>0$,
a positive integer $n_0$ and
a connected rooted graph $H^\bullet$ that contains $G$ as an induced subgraph
such that the following holds
for graphon $W$ with density $p\ge p_0$ such that $t(C_4, W) - p^4 \leq \gamma_0$:
\begin{equation}
t(H^\bullet\oplus K_{m|\ell,n}^\bullet,W)\ge p^{\nume{H}+mn+\ell}
\label{eq:thmuniv}
\end{equation}
for all even integers $m,n,\ell\ge n_0$ such that $m$ is divisible by $5$ and $\ell\ge n + \nume{H}$.
\end{theorem}

\begin{proof}
We first apply Theorem~\ref{thm:core} with $G$ and $p_0$
to get a connected rooted graph $H^\bullet$ that contains $G$ as an induced subgraph and a real $\varepsilon_H>0$.
We prove that the statement of the theorem holds with
\[\gamma_0=\min\left\{\frac{\varepsilon_H^4}{8},\frac{p_0^{160}}{2^{200}}\right\}\quad\mbox{and}\quad n_0=64.\]
Note that Theorem~\ref{thm:core} implies that
\begin{equation}
t(H^\bullet\oplus P_{\ell+1}^\bullet,W)\ge p^{\nume{H}+\ell}
\label{eq:lmcoreeq}
\end{equation}
for every graphon $W$ with density $p\ge p_0$ such that $t(C_4, W) - p^4\leq\gamma_0$, and
for every integer $\ell\ge 8$;
note that $\cut{W-p}\leq\varepsilon_H$ indeed holds for any such graphon by Lemma~\ref{lm:gammagamma}.

Fix even integers $m,n,\ell\ge n_0$ such that $m$ is divisible by $5$ and $\ell\ge n + \nume{H}\ge n_0\ge 64$.
Assume that \eqref{eq:thmuniv} fails for $H^\bullet$, $m$, $n$, $\ell$ and $\gamma_0$, and
let $q$ be the supremum over all $p\in [p_0,1]$ such that
there exists a graphon $W$ with density $p$ and $t(C_4, W) - p^4 \leq \gamma_0$ such that
\begin{equation}
t(H^\bullet\oplus K_{m|\ell,n}^\bullet,W)<p^{\nume{H}+mn+\ell}.
\label{eq:violate}
\end{equation}
Fix such a graphon $W$ with density $p>q/(1+p_0^{mn}/4)$ for the rest of the proof.
Observe that the choices of $q$ and $p$ implies that
every graphon $W'$ with density $p'\ge p(1+p_0^{mn}/4)$ and $t(C_4, W) - (p')^4\leq\gamma_0$
satisfies \eqref{eq:thmuniv}.

We next analyze spectral properties of the operator associated with the graphon $W$
using results presented in Subsection~\ref{subsec:spectral}.
Let $\lambda_i$ be the non-zero eigenvalues of $W$ listed in the non-increasing order of the absolute value (with multiplicities), and
let $f_i$ be the corresponding orthonormal eigenfunctions.
Further,
let $\alpha_i\in [0,\pi/2]$ be the real such that $\langle \jj,f_i\rangle=\cos\alpha_i$, and
set $\delta=1-\cos\alpha_1$.
Finally, let $\gamma=t(C_4,W)-p^4$. Note that $\gamma\le\gamma_0$.

We start with giving a lower bound on $\lambda_1$ in terms of $p$ and $\delta$.
We obtain using \eqref{eq:lambda1big0} and $\gamma\le p_0^4/2^{20}$ that
\begin{equation}
\lambda_1
\ge p(1+2\delta) - 8\delta \gamma^{1/4}
\ge p(1+2\delta) - \delta p_0/4
\ge p(1+7\delta/4).
\label{eq:lambda1big}
\end{equation}
On the other hand,
it holds that $\lambda_1\le p+\gamma/(4p^3)$ by \eqref{eq:lambda1small}.
Hence, we obtain using \eqref{eq:lambda1big} and $\gamma\le p_0^5/2^{20}$ that
\begin{equation}
\delta \le \frac{4}{7}\cdot \frac{\lambda_1-p}{p}\le\frac{\gamma}{7p^3} \le \frac{p_0^2}{2^{20}} \le  \frac{1}{2^{20}}.
\label{eq:dsmallg}
\end{equation}

Let $g_H:[0,1]\to [0,1]$ be defined as $g_H(x)=t_x(H^\bullet,W)$ and
let $g_K:[0,1]\to [0,1]$ be defined as $g_K(x)=t_x(K_{m,n}^\bullet,W)$.
Note that it holds that $g_K(x)=t_x(K_{m,n}^\bullet,W)\le t(K_{m-1,n},W)$ for every $x\in [0,1]$,
which implies that
\begin{equation}
\|g_K\|_2 \leq \| g_K\|_\infty \leq t(K_{m-1,n},W).
\label{eq:gKbound}
\end{equation}
Set $\sigma_i=\langle g_H, f_i\rangle$ and $\kappa_i=\langle g_K, f_i\rangle$.
Since $\|g_H\|_2\le\|g_H\|_\infty\le 1$ and $\|g_K\|_2\le\|g_K\|_\infty\le 1$,
we obtain that $\lvert\sigma_i\rvert\le 1$ and $\lvert\kappa_i\rvert\le 1$ for every $i$.
Next observe that
\begin{equation}
t(H^\bullet\oplus K_{m|\ell,n}^\bullet,W)=\langle g_H, W^{\ell} g_K\rangle=
\sigma_1\kappa_1\lambda_1^{\ell}+\sum_{i\ge 2}\sigma_i\kappa_i\lambda_i^{\ell}.
\label{eq:goalexpression}
\end{equation}

In the next two claims, we give lower bounds on the product $\sigma_1\lambda_1^{\ell}$ and $\kappa_1$.

\begin{claim}
\label{cl:sigma1lower}
For every $\ell'\ge 8$, it holds that
\[\sigma_1 \lambda_1^{\ell'}\ge p^{\nume{H}+\ell'}-2\gamma^{\ell'/4}\delta^{1/2}.\]
\end{claim}

\begin{proof}[Proof of Claim~\ref{cl:sigma1lower}]
Consider a positive integer $\ell'\ge 8$.
The estimate \eqref{eq:lmcoreeq} implies that
\begin{equation}
p^{\nume{H}+\ell'} \le
t(H^\bullet\oplus P_{\ell'+1}^\bullet,W) =
\sigma_1 \lambda_1^{\ell'}(1-\delta)+\sum_{i\ge 2}\sigma_i \lambda_i^{\ell'}\cos\alpha_i.
\label{eq:SidorenkoHl}
\end{equation}
We combine \eqref{eq:SidorenkoHl} with \eqref{eq:powersumgamma} and \eqref{eq:sumcosa} to obtain that
\begin{align*}
p^{\nume{H}+\ell'}
 \le \sigma_1 \lambda_1^{\ell'}(1-\delta)+\sum_{i\ge 2}\sigma_i \lambda_i^{\ell'}\cos\alpha_i 
 & \le \sigma_1 \lambda_1^{\ell'}+\sum_{i\ge 2}\lvert\lambda_i\rvert^{\ell'}\cdot\lvert\cos\alpha_i\rvert \\
 & \le \sigma_1 \lambda_1^{\ell'}+\left(\sum_{i\geq 2} \lambda_i^{2\ell'}\right)^{1/2}\left(\sum_{i\ge 2}\cos^2\alpha_i\right)^{1/2}\\
 & \le \sigma_1 \lambda_1^{\ell'}+2\gamma^{\ell'/4}\delta^{1/2}.
\end{align*}
Hence, we can conclude that
\begin{equation}
\sigma_1 \lambda_1^{\ell'}\ge p^{\nume{H}+\ell'}-2\gamma^{\ell'/4}\delta^{1/2}
\label{eq:sigma1lowerplus}
\end{equation}
holds for every $\ell'\ge 8$.
The estimate given in the statement of the claim now follows.
\end{proof}

\begin{claim}
\label{cl:k1bound1}
It holds that
\[\kappa_1 \geq t(K_{m-1,n},W)((p^4+\gamma)^{n/4}-2\delta^{1/2}).\]
\end{claim}

\begin{proof}[Proof of Claim~\ref{cl:k1bound1}]
We first obtain using the Cauchy-Schwarz inequality that
\[t(K_{m,n}, W)=\langle g_K,\jj\rangle=\sum_{i\ge 1}\kappa_i\cos\alpha_i\le\kappa_1\cos\alpha_1+\left(1-\cos^2\alpha_1\right)^{1/2}\|g_K\|_2,\]
which implies using the definition of $\delta$ and \eqref{eq:gKbound} that
\[
t(K_{m,n}, W)\le\kappa_1\cos\alpha_1+2\delta^{1/2}t(K_{m-1,n},W).
\]
Hence, it follows using Lemma~\ref{lm:K2n2nC4b} that
\[
\kappa_1 \ge t(K_{m,n}, W)-2\delta^{1/2}t(K_{m-1,n},W)
         \ge t(K_{m-1,n},W)(t(K_{{m-1},n},W)^{1/(m-1)}-2\delta^{1/2}).
\]
Since Lemma~\ref{lm:K2n2nC4a} implies that
\[t(K_{{m-1},n},W)^{1/(m-1)}\ge t(C_4,W)^{n/4}=(p^4+\gamma)^{n/4},\]
the statement of the claim now follows.
\end{proof}

We next distinguish the following three cases,
which will be covered in Claims~\ref{cl:univ1}, \ref{cl:univ2} and \ref{cl:univ3}, respectively.
In each of the cases, we show that the fixed graphon $W$ satisfies \eqref{eq:thmuniv},
which would contradict the choice of $W$ and so complete the proof of the theorem.
\begin{enumerate}[label=(\roman*)]
\item It holds that $\delta\le\frac{\gamma^2 n^2 p^{2n}}{1600p^8}$. \label{it:1}
\item It holds that $\delta>\frac{\gamma^2 n^2 p^{2n}}{1600p^8}$ and $\|g_K\|_2\le\frac{\kappa_1}{\gamma^{\ell/8}}$. \label{it:2}
\item It holds that $\delta>\frac{\gamma^2 n^2 p^{2n}}{1600p^8}$ and $\|g_K\|_2>\frac{\kappa_1}{\gamma^{\ell/8}}$. \label{it:3}
\end{enumerate}
Intuitively speaking,
Case~\ref{it:1} corresponds to the situation when the degrees of almost all vertices of $W$ are the same, 
which is equivalent to that $\delta$ is small.
Case~\ref{it:2} corresponds to the situation when the sum in \eqref{eq:goalexpression} is small compared to the leading term,
i.e., the term $\sigma_1\kappa_1\lambda_1^{\ell}$ controls the expression in \eqref{eq:goalexpression}.
In Case~\ref{it:3}, which covers the cases not covered in Cases~\ref{it:1} and~\ref{it:2},
we show that the reason for the sum in \eqref{eq:goalexpression} being large comparatively to the leading term
is that the graphon $W$ contains a part denser than $q$
where we identify sufficinetly many copies of $H^\bullet\oplus K_{m|\ell,n}^\bullet$;
here, we use the choice of $q$ as
the supremum among the densities of graphons violating \eqref{eq:thmuniv} for fixed $m$, $n$ and $\ell$.

\begin{claim}
\label{cl:univ1}
If $\delta$ is at most $\frac{\gamma^2 n^2 p^{2n}}{1600p^8}$,
then \eqref{eq:thmuniv} holds for $m$, $n$, $\ell$ and the graphon $W$.
\end{claim}

\begin{proof}[Proof of Claim~\ref{cl:univ1}]
We start with refining the bound given in Claim~\ref{cl:k1bound1}
using the assumption that $\delta\le\frac{\gamma^2 n^2 p^{2n}}{1600p^8}$,
which is equivalent to $\frac{40\delta^{1/2}p^4}{np^n}\le\gamma$, and
Pascal's inequality:
\begin{align}
\kappa_1 & \ge t(K_{m-1,n},W)\left((p^4+\gamma)^{n/4}-2\delta^{1/2}\right) \nonumber \\
         & \ge t(K_{m-1,n},W)\left(p^n\left(1+\frac{n\gamma}{4p^4}\right)-2\delta^{1/2}\right) \nonumber \\
         & \ge t(K_{m-1,n},W)\left(p^n\left(1+\frac{10\delta^{1/2}}{p^n}\right)-2\delta^{1/2}\right) \nonumber \\
	 & = t(K_{m-1,n},W)\left(p^n+8\delta^{1/2}\right)\ge t(K_{m-1,n},W)p^n(1+8\delta^{1/2}). \label{eq:buffer}
\end{align}
Hence, we conclude using Claim~\ref{cl:sigma1lower}, the bound $\gamma\le p^8/2^{20}$, \eqref{eq:buffer} and $\ell\ge\nume{H}$ that
the first term in \eqref{eq:goalexpression} can be bounded from below as follows:
\begin{align}
\sigma_1\kappa_1\lambda_1^{\ell}
  & \ge \left(p^{\nume{H}+\ell}-2\gamma^{\ell/4}\delta^{1/2}\right)t(K_{m-1,n},W)p^n(1+8\delta^{1/2})\nonumber\\
  & \ge \left(p^{\nume{H}+\ell}-p^{2\ell}\delta^{1/2}/16\right)t(K_{m-1,n},W)p^n(1+8\delta^{1/2})\nonumber\\
  & = t(K_{m-1,n},W)p^n\left(p^{\nume{H}+\ell}+8p^{\nume{H}+\ell}\delta^{1/2}-p^{2\ell}\delta^{1/2}/16-p^{2\ell}\delta/2\right)\nonumber\\
  & \ge t(K_{m-1,n},W)p^n\left(p^{\nume{H}+\ell}+8p^{\nume{H}+\ell}\delta^{1/2}-p^{2\ell}\delta^{1/2}\right)\nonumber\\
  & \ge t(K_{m-1,n},W)p^n\left(p^{\nume{H}+\ell}+8p^{2\ell}\delta^{1/2}-p^{2\ell}\delta^{1/2}\right)
    \ge t(K_{m-1,n},W)p^{\nume{H}+\ell+n}.\label{eq:dom1}
\end{align}
On the other hand,
the absolute value of the sum in \eqref{eq:goalexpression}
can be bounded from above using Cauchy-Schwarz inequality, \eqref{eq:powersumgamma} and \eqref{eq:gKbound} as follows:
\begin{align}
\left\lvert \sum_{i\ge 2}\sigma_i\kappa_i\lambda_i^{\ell}\right\rvert
& \le \left(\sum_{i\ge 2} \kappa_i^2 \sigma_i^2 \right)^{1/2} \left(\sum_{i\ge 2} \lambda_i^{2\ell}\right)^{1/2} \nonumber\\
& \le \left(\sum_{i\ge 1} \kappa_i^2 \sigma_i^2 \right)^{1/2} \left(\sum_{i\ge 1} \lambda_i^{2\ell}\right)^{1/2} \nonumber\\
& = \|g_K\|_2 \cdot \gamma^{\ell/4}
\le t(K_{m-1, n}, W) \gamma^{\ell/4}. \label{eq:error1} 
\end{align}
Hence, we conclude using \eqref{eq:goalexpression}, \eqref{eq:dom1}, \eqref{eq:error1},
Lemma~\ref{lm:K2n2nC4a}, Pascal's inequality, $\gamma\le p^{16}/2^{32}$ and $\ell\ge n + \nume{H} \ge 8$ that
\begin{align*}
t(H^\bullet\oplus K_{m|\ell,n}^\bullet,W)
& \ge t(K_{m-1,n},W)\left(p^{\nume{H}+\ell+n}-\gamma^{\ell/4}\right)\\
& \ge (p^4+\gamma)^{(m-1)n/4}\left(p^{\nume{H}+\ell+n}-\gamma^{1+\ell/8}\right)\\
& \ge  p^{(m-1)n}\left(1+\frac{\gamma}{p^4}\right)^{(m-1)n/4}\left(p^{\nume{H}+\ell+n}-\gamma^{1+\frac{\ell+n + \nume{H}}{16}}\right)\\
& \ge p^{(m-1)n}\left(1+\frac{(m-1)n\gamma}{4p^4}\right)\left(p^{\nume{H}+\ell+n}-\frac{\gamma p^{\nume{H}+\ell+n}}{4}\right)\\
& \ge p^{\nume{H}+\ell+mn}\left(1+\frac{\gamma}{2}\right)\left(1-\frac{\gamma}{4}\right)\ge p^{\nume{H}+\ell+mn},
\end{align*}
which completes the proof of the claim.
\end{proof}

Before stating and proving Claims~\ref{cl:univ2} and~\ref{cl:univ3},
we show that if $\delta$ is large, then $\kappa_1$ is also large as
stated in the next claim.

\begin{claim}
\label{cl:univ23}
If $\delta$ is larger than $\frac{\gamma^2 n^2 p^{2n}}{1600p^8}$, then $\kappa_1\ge p^{mn}$.
\end{claim}

\begin{proof}[Proof of Claim~\ref{cl:univ23}]
We start with a bound on the density of $K_{1,2}$:
\[t(K_{1,2},W) = \langle \jj, W^2\jj \rangle
  = \sum_{i\ge 1} \lambda_i^2\cos^2 \alpha_i
  \geq \lambda_1^2 \cos^2 \alpha_1
  = \lambda_1^2 (1-\delta)^2.\]
Since $mn/10$ is an even integer,
which follows from the facts that $m$ and $n$ are even integers and $m$ is divisible by $5$,
Lemma~\ref{lm:entropy} implies that
\[t(K_{m|mn/10,n},W)\ge\lambda_1^{11mn/10}(1-\delta)^{11mn/10}.\]
On the other hand, we obtain using \eqref{eq:powersumgamma} that
\begin{align*}
t(K_{m|mn/10,n},W) = \langle \jj, W^{mn/10}g_K\rangle
 & = \kappa_1(1-\delta)\lambda_1^{mn/10}+\sum_{i\ge 2}\kappa_i\cos\alpha_i\lambda_i^{mn/10} \\
 & \le \kappa_1\lambda_1^{mn/10}+\sum_{i\ge 2}\lvert\lambda_i\rvert^{mn/10} \\
 & \le \kappa_1\lambda_1^{mn/10}+\gamma^{mn/40}.
\end{align*}
We obtain by combining the just obtained lower and upper bounds on $t(K_{m|mn/10,n},W)$ that
\begin{equation}
 \kappa_1 \ge \lambda_1^{mn}(1-\delta)^{11mn/10}-\gamma^{mn/40}/\lambda_1^{mn/10}.
 \label{eq:kappa1a}
\end{equation}
We next plug the estimate \eqref{eq:lambda1big} to \eqref{eq:kappa1a} and 
we derive using $\gamma\le p^{160}$, Pascal's inequality, and $n_0\ge 20$ that
\begin{align*}
\kappa_1
 & \ge p^{mn}(1+7\delta/4)^{mn}(1-\delta)^{11mn/10}-\gamma^{mn/40}/p^{mn/10}\\
 & \ge p^{mn}\left((1+7\delta/4)^{10}(1-\delta)^{11}\right)^{mn/10}-\gamma^{mn/160}p^{2mn}\\
 & \ge p^{mn}\left((1+70\delta/4)(1-11\delta)\right)^{mn/10}-\gamma^2 p^{2n+mn}\\
 & \ge p^{mn}\left(1+70\delta/4-11\delta-770\delta^2/4\right)^{mn/10}-\gamma^2p^{2n+mn}
\end{align*}
which yields using $\delta\le 2^{-20}$ given by \eqref{eq:dsmallg} that
\[
 \kappa_1 \ge p^{mn}\left(1+4\delta\right)^{mn/10}-\gamma^2p^{2n+mn}
          \ge p^{mn}\left(1+\frac{2\delta mn}{5}-\gamma^2p^{2n}\right).
\]
We now use the assumption from the statement of the claim that $\delta>\gamma^2 n^2 p^{2n}/1600$, and
the fact that $n_0\ge 20$ to derive that
\begin{equation}
 \kappa_1 \ge p^{mn}\left(1+\frac{\gamma^2 mn^3 p^{2n}}{4000}-\gamma^2p^{2n}\right) \ge p^{mn},
 \label{eq:k1stronger}
\end{equation} 
which completes the proof of the claim.
\end{proof}

We are now ready to analyze Cases~\ref{it:2} and~\ref{it:3},
which are covered in the next two claims.

\begin{claim}
\label{cl:univ2}
If $\delta$ is larger than $\frac{\gamma^2 n^2 p^{2n}}{1600p^8}$ and
$\|g_K\|_2$ is at most $\frac{\kappa_1}{\gamma^{\ell/8}}$,
then \eqref{eq:thmuniv} holds for $m$, $n$, $\ell$ and the graphon $W$.
\end{claim}

\begin{proof}[Proof of Claim~\ref{cl:univ2}]
We first plug the assumption $\|g_K\|_2\le\frac{\kappa_1}{\gamma^{\ell/8}}$ to \eqref{eq:goalexpression} and
obtain using \eqref{eq:powersumgamma} and Cauchy-Schwarz inequality that
\begin{align}
t(H^\bullet\oplus K_{m|\ell,n}^\bullet,W)
 & = \sigma_1\kappa_1\lambda_1^{\ell}+\sum_{i\ge 2}\sigma_i\kappa_i\lambda_i^{\ell} \nonumber \\
 & \ge \sigma_1\kappa_1\lambda_1^{\ell}-\sum_{i\ge 2}\lvert\kappa_i\rvert\cdot\lvert\lambda_i\rvert^{\ell} \nonumber \\
 & \ge \sigma_1\kappa_1\lambda_1^{\ell}-\|g_K\|_2\gamma^{\ell/4} \ge \kappa_1\left(\sigma_1\lambda_1^{\ell}-\gamma^{\ell/8}\right).
 \label{eq:case2a}
\end{align}
We next estimate $\sigma_1\lambda_1^{\ell}$
using Claim~\ref{cl:sigma1lower} with $\ell'=\ell-1$, \eqref{eq:lambda1big} and $\ell\ge 10$,
as follows:
\begin{align}
\sigma_1\lambda_1^{\ell}=\sigma_1\lambda_1^{\ell-1}\lambda_1
 & \ge \left(p^{\nume{H}+\ell-1}-2\gamma^{(\ell-1)/4}\delta^{1/2}\right)p(1+7\delta/4) \nonumber \\
 & \ge p^{\nume{H}+\ell}\left(1-\frac{2\gamma^{(\ell-1)/4}\delta^{1/2}}{p^{\nume{H}+\ell-1}}\right)(1+7\delta/4) \nonumber \\
 & \ge p^{\nume{H}+\ell}\left(1-2\gamma^{\ell/8}\delta^{1/2}\right)(1+7\delta/4).
 \label{eq:case2b}
\end{align}
We obtain using Claim~\ref{cl:univ23}, \eqref{eq:case2a}, \eqref{eq:case2b}, $\gamma\le p^{64}$ and $\ell\ge n+\nume{H}$ that
\begin{align}
t(H^\bullet\oplus K_{m|\ell,n}^\bullet,W)
 & \ge p^{mn}\left(\sigma_1\lambda_1^{\ell}-\gamma^{\ell/8}\right)\nonumber\\
 & \ge p^{mn}\left(p^{\nume{H}+\ell}\left(1-2\gamma^{\ell/8}\delta^{1/2}\right)(1+7\delta/4)-\gamma^{\ell/8}\right)\nonumber\\
 & \ge p^{\nume{H}+\ell+mn}\left(\left(1-\gamma^{\ell/16}\delta^{1/2}/2^{10}\right)(1+7\delta/4)-\gamma^{\ell/16}\right).
 \label{eq:case2c}
\end{align}
We next estimate $\gamma^{\ell/16}$ using that 
the assumption that $\frac{\gamma^2 n^2 p^{2n}}{1600p^8}<\delta$,
$n\ge 64$ and $\ell\ge n$ (and $\ell\ge 64$) as follows:
\[\gamma^{\ell/16}
 \le \gamma^2 \gamma^{\ell/32}
 \le \frac{1600\delta}{n^2 p^{2n}} p^{160\ell/32}
 \le \frac{\delta}{p^{2n}} p^{160n/32}
 \le \delta.\]
Hence, we obtain from \eqref{eq:case2c} that
\begin{align*}
t(H^\bullet\oplus K_{m|\ell,n}^\bullet,W)
& \ge p^{\nume{H}+\ell+mn}\left(\left(1-\delta^{3/2}/2^{10}\right)(1+7\delta/4)-\delta\right) \\
& \ge p^{\nume{H}+\ell+mn}\left(1+7\delta/4-\delta^{3/2}/2^{10}-7\delta^{5/2}/2^{12}-\delta\right) \ge p^{\nume{H}+\ell+mn},
\end{align*}
which yields the statement of the claim.
\end{proof}

\begin{claim}
\label{cl:univ3}
If $\delta$ is larger than $\frac{\gamma^2 n^2 p^{2n}}{1600p^8}$ and
$\|g_K\|_2$ is larger than $\frac{\kappa_1}{\gamma^{\ell/8}}$,
then \eqref{eq:thmuniv} holds for $m$, $n$, $\ell$ and the graphon $W$.
\end{claim}

\begin{proof}[Proof of Claim~\ref{cl:univ3}]
Our aim is to show that $W$ contains a significantly denser part and
use this part to obtain a contradiction with the choice of $q$ at the beginning of the proof.
Observe that 
we can assume the eigenfunction $f_1$ is non-negative since the graphon $W$ is non-negative (if needed,
change $f_1$ on a set of measure zero).
We define subsets $S$ and $S'$ of $[0,1]$ (also see Figure~\ref{fig:SS}) as follows:
\begin{align*}
S & = \{x\in [0,1]\mbox{ such that }f_1(x)\le\gamma^{\ell/100}\}\mbox{, and}\\
S' & = \{x\in [0,1]\mbox{ such that }f_1(x)\le 1-\delta^{1/4}\}.
\end{align*}
Note that $\gamma^{\ell/100}\le\gamma_0^{n_0/100}\le 2^{-128}$ by the choice of $n_0$ and $\gamma_0$ and
$\delta^{1/4}\le 1/32$ by \eqref{eq:dsmallg}, and
so it holds that $S\subseteq S'$.

We start with bounding the measures of the sets $S$ and $S'$ from above.
Since $\|f_1\|_1=\langle f_1,\jj\rangle=1-\delta$ and $0\le f_1(x)\le 1$ for all $x\in [0,1]$,
we obtain that $\mu(S)\gamma^{\ell/100}+1-\mu(S)\ge 1-\delta$,
which implies using \eqref{eq:dsmallg} that
\begin{equation}
\mu(S) \leq \frac{\delta}{1-\gamma^{\ell/100}}\le \frac{p_0}{2^{20}(1-\gamma^{\ell/100})}\le \frac{p}{2^{16}}.
\label{eq:Ssmall}
\end{equation}
Similarly, it holds that $\mu(S')(1-\delta^{1/4})+1-\mu(S')\ge 1-\delta$,
which implies that
\begin{equation}
\mu(S') \leq \delta^{3/4}.
\label{eq:Sprimesmall}
\end{equation}

Our next step is proving the following lower bound on the measure of the set $S$:
\begin{equation}
\mu(S) \geq \left(1-{\gamma^{\ell/20}}\right)p^{mn} \geq \frac{p^{mn}}{2}. 
\label{eq:Slarge}
\end{equation}
Suppose that \eqref{eq:Slarge} does not hold, i.e., $\mu(S)<\left(1-{\gamma^{\ell/20}}\right)p^{mn}$.
We first bound $\kappa_1$ as follows (we use the Sidorenko property of $K_{m,n}$,
i.e., $t(K_{m,n},W)\ge p^{mn}$, in the penultimate inequality):
\begin{align}
\kappa_1
 & = \int_{[0,1]}f_1(x)g_K(x)\dd x \nonumber \\
 & \ge \int_{[0,1]}\gamma^{\ell/100}g_K(x)\dd x+\int_{S}\left(f_1(x)-\gamma^{\ell/100}\right)g_K(x)\dd x \nonumber \\
 & \ge \gamma^{\ell/100}\int_{[0,1]}g_K(x)\dd x-\gamma^{\ell/100}\mu(S) \nonumber \\
 & > \gamma^{\ell/100}\left(t(K_{m,n},W)-\left(1-\gamma^{\ell/20}\right)p^{mn}\right) \nonumber \\
 & \ge \gamma^{\ell/100} t(K_{m,n},W) \left(1-1+\gamma^{\ell/20}\right) \ge \gamma^{\ell/16} t(K_{m,n},W).
 \label{eq:k1ub}
\end{align} 
On the other hand, we obtain using the assumption $\|g_K\|_2>\frac{\kappa_1}{\gamma^{\ell/8}}$ that
\begin{align*}
\kappa_1
 & < \gamma^{\ell/8}\|g_K\|_2 = \gamma^{\ell/8}\left(\int_{[0,1]}g_K(x)^2\dd x\right)^{1/2} \\
 & \le \gamma^{\ell/8}\|g_K\|_1^{1/2}\|g_K\|_\infty^{1/2} \\
 & \le \gamma^{\ell/8}t(K_{m,n},W)^{1/2}t(K_{m-1,n},W)^{1/2},
\end{align*}
which implies using Lemma~\ref{lm:K2n2nC4b}, the Sidorenko property of $K_{m,n}$, $\gamma\le p^{16}$ and $\ell\ge n+\nume{H}$ that
\begin{align}
\kappa_1
 & < \gamma^{\ell/8}t(K_{m,n},W)^{1/2}t(K_{m,n},W)^{(m-1)/2m} \nonumber \\
 & \le \gamma^{\ell/16} p^n t(K_{m,n},W)^{(2m-1)/2m} \le \gamma^{\ell/16}t(K_{m,n},W).
 \label{eq:k1lb}
\end{align}
However, the estimates \eqref{eq:k1ub} and \eqref{eq:k1lb} contradict each other.
Hence, we conclude that \eqref{eq:Slarge} holds.

\begin{figure}
\begin{center}
\epsfbox{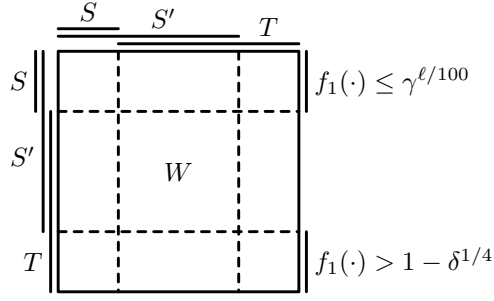}
\end{center}
\caption{The illustration of the definition of the sets $S$, $S'$ and $T$
         from the proof of Claim~\ref{cl:univ3}.}
\label{fig:SS}
\end{figure}

We next analyze the graphon $W[T]$ where $T=[0,1]\setminus S$ and
start with estimating the density of $W[T]$;
we refer to Figure~\ref{fig:SS} for the illustration of the relation of the sets $S$, $S'$ and $T$.
As the first step to estimate the density of $W[T]$,
we bound the density of $W$ between the sets $S$ and $[0,1]\setminus S'$ (we refer to Figure~\ref{fig:SS} and
recall that $\lambda_1\in [p_0,1]$):
\begin{align*}
\mu(S)\gamma^{\ell/100} & \ge \int_S f_1(x)\dd x
                       = \int_S \frac{1}{\lambda_1}\left(\int_{[0,1]}W(x,y)f_1(y)\dd y\right)\dd x
                       \ge \int_{S\times [0,1]}W(x,y)f_1(y)\dd x\dd y\\
		     & \ge \int_{S\times([0,1]\setminus S')}W(x,y)f_1(y)\dd x\dd y
		       \ge (1-\delta^{1/4})\int_{S\times([0,1]\setminus S')}W(x,y)\dd x\dd y,
\end{align*}
which yields that
\begin{equation}
\int_{S\times([0,1]\setminus S')}W(x,y)\dd x\dd y\le\frac{\gamma^{\ell/100}\mu(S)}{1-\delta^{1/4}}\le 2\gamma^{\ell/100}\mu(S).
\label{eq:fact23}
\end{equation}
Hence, we can estimate the density of $W$ on $T^2$ using \eqref{eq:fact23} as follows (we again refer to Figure~\ref{fig:SS}):
\begin{align*}
\int_{T^2}W(x,y)\dd x\dd y
  & = \int_{[0,1]^2}W(x,y)\dd x\dd y-\int_{S^2}W(x,y)\dd x\dd y-2\int_{S\times ([0,1]\setminus S)}W(x,y)\dd x\dd y \\
  & \ge p - \mu(S)^2 - 2\mu(S)\cdot\mu(S'\setminus S) - \int_{S\times([0,1]\setminus S')}W(x,y)\dd x\dd y\\
  & \ge p - \mu(S)^2 - 2\mu(S)\cdot\mu(S'\setminus S) - 4\gamma^{\ell/100}\lambda_1\mu(S) \\
  & \ge p - \mu(S)^2 - 2\mu(S)\cdot\mu(S') - 4\gamma^{\ell/100}\lambda_1\mu(S),
\end{align*}
which combines with estimates \eqref{eq:Ssmall}, \eqref{eq:Sprimesmall}, $\delta\le p^2/2^{20}$ by \eqref{eq:dsmallg}, and
$\gamma^{\ell/100}\le\gamma_0^{n_0/100}\le p^{100}/2^{128}$ by the choice of $n_0$ and $\gamma_0$ to
\begin{align*}
\int_{T^2}W(x,y)\dd x\dd y
  & \ge p - \frac{p\mu(S)}{2^{16}} - 2\mu(S)\delta^{3/4} - 4\gamma^{\ell/100}\mu(S) \\
  & \ge p - \frac{p\mu(S)}{2^{16}} - \frac{p\mu(S)}{2^{14}} - \frac{p\mu(S)}{2^{126}}
    \ge p\left(1-\frac{\mu(S)}{2^{10}}\right).
\end{align*}
Hence, the density $p_T$ of the graphon $W[T]$ is at least
\begin{equation}
p_T \ge \frac{\int_{T^2}W(x,y)\dd x\dd y}{(1-\mu(S))^2}
    \ge p \left(1-\frac{\mu(S)}{2^{10}}\right)(1+2\mu(S)) \ge p\left(1+\frac{511\mu(S)}{256}\right),
\label{eq:densWT}
\end{equation}
which combines with \eqref{eq:Slarge} to
\begin{equation}
p_T \ge p\left(1+\frac{511\mu(S)}{512}\right) > p\left(1+p^{mn}/4\right).
\label{eq:densWTq}
\end{equation}
Note that if $q$ were equal to $1$,
then \eqref{eq:densWTq} would imply that $p_T>1$ (recall that $q\le p(1+p_0^{mn}/4)$),
which is impossible, and so it holds that $q<1$.
Finally, we estimate $\gamma_T=t(C_4,W[T])-p_T^4$ as follows:
\begin{align}
\gamma_T = t(C_4,W[T])-p_T^4
         & \le \frac{t(C_4,W)}{(1-\mu(S))^4}-p^4\left(1+\frac{511\mu(S)}{256}\right)^4\nonumber \\
	 & \le (p^4+\gamma)(1+5\mu(S))-p^4\left(1+\frac{511\mu(S)}{256}\right)^4\nonumber \\
	 & \le \gamma+5\gamma\mu(S)+5p^4\mu(S)-\frac{511p^4\mu(S)}{64}\nonumber \\
	 & \le \gamma+\frac{5p^4\mu(S)}{2^{200}}+5p^4\mu(S)-\frac{511p^4\mu(S)}{64}\le\gamma.
\label{eq:gammaWT}
\end{align}
Since the density $p_T$ of $W[T]$ is larger than $q$ by \eqref{eq:densWTq} (again recall that $q\le p(1+p_0^{mn}/4)$) and
$\gamma_T\le\gamma_0$ by \eqref{eq:gammaWT},
the choice of $q$ implies that
\[t(H^\bullet\oplus K_{m|\ell,n}^\bullet,W[T]) \ge p_T^{\nume{H}+\ell+mn}.\]
Hence, we obtain using \eqref{eq:tWh}, \eqref{eq:Ssmall} and \eqref{eq:densWT} that
\begin{align*}
t(H^\bullet\oplus K_{m|\ell,n}^\bullet,W)
 & \ge (1-\mu(S))^{\numv{H}+\ell+m+n-2}p_T^{\nume{H}+\ell+mn} \\
 & \ge p^{\nume{H}+\ell+mn}(1-\mu(S))^{\numv{H}+\ell+m+n-2}\left(1+\frac{511\mu(S)}{256}\right)^{\nume{H}+\ell+mn} \\
 & \ge p^{\nume{H}+\ell+mn}\left(\left(1-\mu(S)\right)\left(1+\frac{511\mu(S)}{256}\right)\right)^{\numv{H}+\ell+m+n-2} \\
 & = p^{\nume{H}+\ell+mn}\left(1-\mu(S)+\frac{511\mu(S)}{256}-\frac{511\mu(S)^2}{256}\right)^{\numv{H}+\ell+m+n-2} \\
 & \ge p^{\nume{H}+\ell+mn}\left(1-\mu(S)+\frac{511\mu(S)}{256}-\frac{511\mu(S)}{2^{24}}\right)^{\numv{H}+\ell+m+n-2} \ge p^{\nume{H}+\ell+mn},
\end{align*} 
which implies that $W$ satisfies \eqref{eq:violate} under the assumptions made in the statement of the claim.
\end{proof}
Claims~\ref{cl:univ1}, \ref{cl:univ2} and \ref{cl:univ3} imply that $W$ satisfies \eqref{eq:violate},
which contradicts the choice of $W$ as violating \eqref{eq:violate}, and
so completes the proof of the theorem.
\end{proof}

Theorem~\ref{thm:universal}
together with the Counting Lemma (Lemma~\ref{lm:cutdistance}) applied for $C_4$ 
yields the following.

\begin{corollary}
\label{cor:universal}
For every graph $G$ with girth at least $50$ and 
every real number $p_0 \in (0,1)$,
there exist a real $\varepsilon_0>0$,
a positive integer $n_0$ and
a connected rooted graph $H^\bullet$ that contains $G$ as an induced subgraph
such that the following holds
for every graphon $W$ with density $p \geq p_0$ such that $\cut{W-p} \leq \varepsilon_0$:
\[t(H^\bullet\oplus K_{m|\ell,n}^\bullet,W)\ge p^{\nume{H}+mn+\ell}\]
for all even integers $m,n,\ell\ge n_0$ such that $m$ is divisible by $5$ and $\ell\ge n + \nume{H}$.
\end{corollary}

\section{Non-local regime}
\label{sec:nonlocal}

In this section, we present the result covering the non-local regime,
i.e., the regime when one of the $k$ color classes is far from being quasirandom with density $1/k$.
We start with the lemma that informally says that
if the density of a graph $H$ in a given graphon is small,
then the graphon contains a large sparse part.
Since it is folklore that complete graphs
satisfy the Kohayakawa-Nagle-R\"odl-Schacht Conjecture (Conjecture~\ref{conj:KNRS} in Section~\ref{sec:concl}),
a simple proof of this fact can be found in~\cite[Theorem 4.2]{Lee21},
the next lemma holds for $H$ being a complete graph and so it holds for all graphs $H$.
We include a self-contained proof in Appendix for completeness.

\begin{lemma}
\label{lm:omegaalpha}
For every graph $H$ and real $\delta\in (0,1)$
there exist reals $\omega_0>0$ and $\alpha_0>0$ such that 
every graphon $W$ satisfies at least one of the following:
\begin{itemize}
\item $t(H, W) \geq \omega_0$, or
\item $\alpha_{\delta}(W)\ge\alpha_0$.
\end{itemize}
\end{lemma}

The next theorem is the main result of this section.
The theorem asserts that
if a graphon $W$ is far from the quasirandom graphon and does not contain a sparse part,
then the density of $H^\bullet\oplus K_{m|\ell,n}^\bullet$ in $W$ is
at least the density of $H^\bullet\oplus K_{m|\ell,n}^\bullet$ in the quasirandom graphon with the same density.

\begin{theorem}
\label{thm:bip}
For all positive reals $p_0\in (0,1)$, $\delta\in (0,1)$ and $\varepsilon_0\in (0,1)$, and
rooted graph $H^{\bullet}$,
there exist a positive real $\alpha_0$ and a positive integer $n_0$ such that
every graphon $W$ with density $p\ge p_0$ and with $\cut{W-p}\ge\varepsilon_0$
satisfies at least one of the following:
\begin{itemize}
\item $\alpha_{\delta}(W)\ge\alpha_0$, or
\item it holds for all integers $m,n\ge n_0$ and $\ell\le mn/4$ that
      \[t\left(H^\bullet\oplus K_{m|\ell,n}^\bullet,W\right)\ge p^{\nume{H}+mn+\ell}.\]
\end{itemize}
\end{theorem}

\begin{proof}
Fix the reals $p_0$, $\delta$ and $\varepsilon_0$ and the rooted graph $H^{\bullet}$.
Let $\omega_0$ and $\alpha'_0$ be the reals obtained by applying Lemma~\ref{lm:omegaalpha} with $\delta$ and $H$, and
let $d_0=\omega_0(p_0\varepsilon_0/10)^{16\numv{H}}$.
Choose $n_0\ge 2$ such that
\[d_0\left(1+10^{-11}\varepsilon_0^{16}\right)^{n_0^2/4}\ge 1.\]
We show that the statement of the theorem holds
for this choice of $n_0$ and for $\alpha_0=\alpha'_0 p_0 (\varepsilon_0/10)^{16}$.

Fix a graphon $W$ and
let $A$ be the set of $x\in [0,1]$ such that $t_x(H^{\bullet},W)\le d_0$.
First suppose that $\mu(A)>p_0(\varepsilon_0/10)^{16}$ and
note that $t(H,W[A])\le d_0/\mu(A)^{\numv{H}}<\omega_0$.
Hence, the graphon $W[A]$ does not satisfy the first conclusion of Lemma~\ref{lm:omegaalpha} and
so it satisfies the second conclusion, i.e., $\alpha_{\delta}(W[A])\ge\alpha'_0$.
It follows that $\alpha_{\delta}(W)\ge\alpha'_0\mu(A)\ge\alpha_0$.
Therefore, we can assume in the rest of the proof of the theorem that $\mu(A)\le p_0(\varepsilon_0/10)^{16}$.

Let $W'$ be the graphon obtained from the graphon $W$ by setting 
\[W'(x,y)=\begin{cases}
          0 & \mbox{if $x\in A$ or $y\in A$, and}\\
	  W(x,y) & \mbox{otherwise.}
	  \end{cases}\]
Note that the density $p'$ of $W'$ is at least $p-3\mu(A)\ge p(1-3(\varepsilon_0/10)^{16})$ and
\[\cut{W-p'}\ge\cut{W-p}-\lvert p-p'\rvert\ge\cut{W-p}-3(\varepsilon_0/10)^{16}.\]
As $\cut{W-p}\ge\varepsilon_0$, we obtain using Lemma~\ref{lm:Kab} (note that we need that $\ell\le mn/4$ to apply the lemma) that
\begin{align}
t(K_{m|\ell,n},W')
& \ge p^{mn+\ell}\left(1-3(\varepsilon_0/10)^{16}\right)^{mn+\ell}\left(1+10^{-9}\left(1-3\varepsilon_0^{15}/10^{16}\right)^{16}\varepsilon_0^{16}\right)^{mn/4+\ell/4} \nonumber\\
& \ge p^{mn+\ell}\left(1-12(\varepsilon_0/10)^{16}\right)^{mn/4+\ell/4}\left(1+10^{-10}\varepsilon_0^{16}\right)^{mn/4+\ell/4} \nonumber\\
& \ge p^{mn+\ell}\left(\left(1-10^{-14}\varepsilon_0^{16}\right)\left(1+10^{-10}\varepsilon_0^{16}\right)\right)^{mn/4+\ell/4} \nonumber\\
& \ge p^{mn+\ell}\left(1+10^{-11}\varepsilon_0^{16}\right)^{mn/4+\ell/4}. \label{eq:tKmlnW}
\end{align}
We next estimate the density of $H^\bullet\oplus K_{m|\ell,n}^\bullet$ in the graphon $W$ as follows:
\begin{align*}
t(H^\bullet\oplus K_{m|\ell,n}^\bullet,W) & = \int_{[0,1]} t_x(H^\bullet, W) t_x(K_{m|\ell,n}^\bullet, W)\dd x \\
                                          & \ge d_0 \int_{[0,1]\setminus A} t_x(K_{m|\ell,n}^\bullet,W) \dd x 
					    \ge d_0 \int_{[0,1]\setminus A} t_x(K_{m|\ell,n}^\bullet,W') \dd x
					    = d_0 t(K_{m|\ell,n},W').
\end{align*}
Hence, we obtain using \eqref{eq:tKmlnW} and the choice of $n_0$ that
\[
t(H^\bullet\oplus K_{m|\ell,n}^\bullet,W) 
  \ge d_0 p^{mn+\ell}\left(1+10^{-11}\varepsilon_0^{16}\right)^{mn/4} 
  \ge p^{mn+\ell}\ge p^{\nume{H}+mn+\ell}.
\]
The statement of the theorem now follows.
\end{proof}

\section{Main result}
\label{sec:main}

We are now ready to prove our main results---Theorem~\ref{thm:main} and more general Theorem~\ref{thm:maink}.
Since there are graphs with arbitrarily large chromatic number and girth at least $50$
by a classical result of Erd\H os~\cite{Erd59}, also see~\cite{AloS16},
the former is an immediate corollary of the next theorem,
which is a special case of Theorem~\ref{thm:step} proven later in this section.
However, we include a short proof of the next theorem to highlight the main steps of the argument in the case of $2$ colors.

\begin{theorem}
\label{thm:pair}
For every graph $G$ with girth at least $50$,
there exist a positive integer $N_0$ and
a connected rooted graph $H^\bullet$ that contains $G$ as an induced subgraph
such that the following holds for any graphon $W$:
\[
t\left(H^\bullet \oplus K_{m|n+\nume{H}, n}^\bullet,W\right)+t\left(H^\bullet \oplus K_{m|n+\nume{H}, n}^\bullet,1-W\right) \geq 
2^{1-\nume{H^\bullet \oplus K_{m|n+\nume{H}, n}^\bullet}}  = 2^{1-2\nume{H}-(m+1)n}
\]
for all even integers $m,n\ge N_0$ such that $m$ is divisible by $5$.
\end{theorem}

\begin{proof}
Fix a graph $G$ and an integer $k$ as in the statement.
We apply Corollary~\ref{cor:universal} with $p_0=1/4$
to obtain a positive real $\varepsilon_0$, a positive integer $n_0$ and
a connected rooted graph $H^\bullet$ containing $G$ as an induced subgraph such that
every graphon $W$ with density $p \geq 1/4$ and with $\cut{W-p} \leq \varepsilon_0$ satisfies that
\[t\left(H^\bullet\oplus K_{m|n+\nume{H},n}^\bullet,W\right)\ge p^{2\nume{H}+(m+1)n}\]
for all even integers $m,n\ge n_0$ such that $m$ is divisible by $5$.
We next apply Theorem~\ref{thm:bip} with $p_0=1/4$, $\delta=\min\{\varepsilon_0/2,1/16\}$, $\varepsilon_0$ and $H^\bullet$
to obtain a positive real $\alpha_0$ and an integer $n_0$ with the properties given in Theorem~\ref{thm:bip}.
We now set $N_0$ in a way that $N_0$ is at least $\max\{8,n_0,\nume{H}\}$ and 
the following holds for all $m,n\ge N_0$:
\begin{align}
\alpha_0^{\numv{H}+\nume{H}+m+2n-1} & \ge 1.5^{1-2\nume{H}-(m+1)n}\label{eq2:choiceN0}.
\end{align}

We now show that the statement of theorem holds for any graphon $W$ and
all even integers $m,n\ge N_0$ such that $m$ is divisible by $5$.
To do so, fix a graphon $W$.
By symmetry, we can assume that the density $p$ of $W$ is at least $1/2$ (otherwise,
we consider the graphon $1-W$ instead of $W$).

Suppose that $\alpha_\delta(1-W)\ge\alpha_0$, and
let $h:[0,1]\to [0,1]$ be such that $\lvert h\rvert_1\ge\alpha_0$ and
the density of $(1-W)[h]=1-W[h]$ is at most $\delta$.
Let $q$ be the density of $W[h]$; note that $q\ge 1-\delta$.
The triangle inequality implies that
$\cut{W[h]-q}\le\cut{1-W[h]}+(1-q)\le 2\delta\le\varepsilon_0$.
Hence, Corollary~\ref{cor:universal}, \eqref{eq:tWh} and \eqref{eq2:choiceN0} yield that
\begin{align*}
t\left(H^\bullet\oplus K_{m|n+\nume{H},n}^\bullet,W\right)
& \ge \alpha_0^{\numv{H}+\nume{H}+m+2n-1}t\left(H^\bullet\oplus K_{m|n+\nume{H},n}^\bullet,W[h]\right)\\
& \ge \alpha_0^{\numv{H}+\nume{H}+m+2n-1}q^{2\nume{H}+(m+1)n} \\
& \ge 1.5^{1-2\nume{H}-(m+1)n}(1-\delta)^{2\nume{H}+(m+1)n}\ge 2^{1-2\nume{H}-(m+1)n}.
\end{align*}
We conclude that if $\alpha_\delta(1-W)\ge\alpha_0$ or (symmetrically) $\alpha_\delta(W)\ge\alpha_0$,
then the statement of the theorem holds.
In the rest of the proof,
we assume that both $\alpha_\delta(W)$ and $\alpha_\delta(1-W)$ are at most $\alpha_0$.

We now derive from Corollary~\ref{cor:universal} if $\cut{W-p}\le\varepsilon_0$ or
from Theorem~\ref{thm:bip} if $\cut{W-p}>\varepsilon_0$ that
\begin{equation}
t\left(H^\bullet\oplus K_{m|n+\nume{H},n}^\bullet,W\right)\ge p^{2\nume{H}+(m+1)n}.\label{eq:color1}
\end{equation}
In particular,
if $p\ge 3/4$, it follows that
\[
t\left(H^\bullet\oplus K_{m|n+\nume{H},n}^\bullet,W\right)\ge 1.5^{2\nume{H}+(m+1)n}\cdot 2^{-2\nume{H}-(m+1)n}\ge 2^{1-2\nume{H}-(m+1)n},
\]
and so the statement of the theorem holds.
On the other hand,
if $p<3/4$, then the density of $1-W$ is at least $p_0=1/4$, and
we can derive from Corollary~\ref{cor:universal} or from Theorem~\ref{thm:bip},
depending whether the cut distance between $1-W$ and $1-p$ is at most $\varepsilon_0$ or larger,
that
\begin{equation}
t\left(H^\bullet\oplus K_{m|n+\nume{H},n}^\bullet,1-W\right)\ge (1-p)^{2\nume{H}+(m+1)n}.\label{eq:color2}
\end{equation}
We combine \eqref{eq:color1} and \eqref{eq:color2} to obtain that
\[t\left(H^\bullet \oplus K_{m|n+\nume{H}, n}^\bullet,W\right)+t\left(H^\bullet \oplus K_{m|n+\nume{H}, n}^\bullet,1-W\right)
  \ge p^{2\nume{H}+(m+1)n}+(1-p)^{2\nume{H}+(m+1)n},\]
which is at least $2^{1-2\nume{H}-(m+1)n}$ by convexity of the function $x^{2\nume{H}+(m+1)n}$.
This completes the proof of the theorem.
\end{proof}

We are now ready to prove the more general result, which implies Theorem~\ref{thm:maink}.
To display the dependance on the length of the path between $H^\bullet$ and $K_{m,n}^\bullet$,
we decided to use $\ell$ throughout the proof to denote the length of this path.

\begin{theorem}
\label{thm:step}
For every graph $G$ with girth at least $50$ and every integer $k\ge 2$,
there exist a positive integer $N_k$ and
a connected rooted graph $H^\bullet$ that contains $G$ as an induced subgraph
such that the following holds
for any $k$ graphons $W_1, \dots, W_k$ such that $W_1+\cdots+W_k=1$:
\[
\sum_{i=1}^k t\left(H^\bullet \oplus K_{m|\ell, n}^\bullet, W_i\right) \geq k^{1-\nume{H}-mn-\ell}
\] 
for all even integers $m,n\ge N_k$ and $\ell=n+\nume{H}$ such that $m$ is divisible by $5$.
\end{theorem}

\begin{proof}
Fix the graph $G$ and the integer $k$.
Apply Corollary~\ref{cor:universal} with $p_0=1/2k$
to obtain a positive real $\varepsilon_0$, a positive integer $n_0$ and
a connected rooted graph $H^\bullet$ containing $G$ as an induced subgraph such that
every graphon $W$ with density $p \geq 1/2k$ and with $\cut{W-p} \leq \varepsilon_0$ satisfies that
\begin{equation}
t\left(H^\bullet\oplus K_{m|\ell,n}^\bullet,W\right)\ge p^{\nume{H}+mn+\ell}\label{eq:coruniv}
\end{equation}
for all even integers $m,n,\ell\ge n_0$ such that $m$ is divisible by $5$ and $\ell\ge n + \nume{H}$.
Without loss of generality, we may assume that $\varepsilon_0\le 1/(4k)$.

We next repeatedly apply Theorem~\ref{thm:bip}. 
Set $\delta_1=\varepsilon_0/2\le 1/8$.
For $k'=2,\ldots,k$,
apply Theorem~\ref{thm:bip} with $p_0=1/2k$, $\delta=\delta_{k'-1}/2$, $\varepsilon_0$ and $H^\bullet$
to obtain a positive real $\alpha_{k'}$ and an integer $n_{k'}$, and
set $\delta_{k'}=\min\{\delta_{k'-1}\alpha_{k'}^2/2,1/4k'\}$.
Finally, set $N_k$ in a way that $N_k$ is at least $n_0$ and the following holds for all $m,n\ge N_k$:
\begin{align}
n + \nume{H} & \le mn/4,\label{eq:choiceNkL}\\
\left(1+1/4k\right)^{2\nume{H}+(m+1)n} & \ge k, \label{eq:choiceNkexp}
\end{align}
and
\begin{equation}
  \alpha_{k'}^{\numv{H}+\nume{H}+m+2n-1}
  \left(\frac{1-\delta_{k'-1}}{k'-1}\right)^{2\nume{H}+(m+1)n}\ge
  \left(\frac{1}{k'}\right)^{2\nume{H}+(m+1)n-1}\label{eq:choiceNk}
\end{equation}  
for $k'=2,\ldots,k$.
Note that such $N_k$ exists
since it holds $\frac{1-\delta_{k'-1}}{k'-1}\ge\frac{4k'-1}{4(k'-1)k'}>\frac{1}{k'}$ for $k'=2,\ldots,k$.

We now prove the following claim, which implies that the statement of the theorem.

\begin{claim}
\label{cl:step}
Let $k'\in [k]$.
Any $k'$ graphons $W_1, \dots, W_{k'}$
such that $\|W_1 + \cdots +W_{k'}\|_\infty\le 1$ and $t\left(K_2, W_1 + \cdots +W_{k'}\right) \geq 1-\delta_{k'}$
satisfy that
\begin{equation}
\sum_{i=1}^{k'} t\left(H^\bullet \oplus K_{m|\ell, n}^\bullet, W_i\right) \geq
  k'\cdot\left(\frac{t\left(K_2, W_1 + \dots +W_{k'}\right)}{k'}\right)^{\nume{H}+mn+\ell}\label{eq:mainind}
\end{equation}
for all even integers $m,n\ge N_k$ and $\ell=n+\nume{H}$ such that $m$ is divisible by $5$.
\end{claim}

\begin{proof}[Proof of Claim~\ref{cl:step}]
The proof proceed by induction on $k'=1,\ldots,k$.
The base of the induction is the case $k'=1$.
Consider a graphon $W_1$ such that $t(K_2,W_1)\geq 1-\delta_1=1-\varepsilon_0/2$,
i.e., the density $p$ of $W_1$ is at least $1-\varepsilon_0/2$.
Note that $\cut{W_1-p}\le\cut{W_1-1}+\lvert 1-p\rvert\le\varepsilon_0$.
Hence, Corollary~\ref{cor:universal} implies that \eqref{eq:coruniv} holds for $W=W_1$,
which is equivalent to \eqref{eq:mainind} for $k'=1$.

We now present the induction step.
Consider $k'\in\{2,\ldots,k\}$ and assume that the claim holds for $k'-1$.
Consider graphons $W_1,\ldots,W_{k'}$ and let $p_1,\ldots,p_{k'}$ be their respective densities;
note that $p_1+\cdots+p_{k'}\geq 1-\delta_{k'}$.
First suppose that there exists $i$ such that $\alpha_{\delta_{k'-1}/2}(W_i)\ge\alpha_{k'}$.
By symmetry, we may assume that $i=k'$.
Hence, there exists $h:[0,1]\to [0,1]$ with $\|h\|_1\ge\alpha_{k'}$ such that
the density of $W_{k'}[h]$ is at most $\delta_{k'-1}/2$.
It follows that the density of $(W_1+\cdots+W_{k'-1})[h]$ is at least
\[1-\frac{\delta_{k'}}{\|h\|_1^2}-\frac{\delta_{k'-1}}{2}\ge
  1-\frac{\delta_{k'}}{\alpha_{k'}^2}-\frac{\delta_{k'-1}}{2}\ge
  1-\frac{\delta_{k'-1}}{2}-\frac{\delta_{k'-1}}{2}=1-\delta_{k'-1}.\]
By the induction hypothesis, it holds that
\[\sum_{i=1}^{k'-1} t(H^\bullet \oplus K_{m|\ell, n}^\bullet, W_i[h]) \geq
  \frac{(1-\delta_{k'-1})^{\nume{H}+mn+\ell}}{(k'-1)^{\nume{H}+mn+\ell-1}}.\]
It follows using \eqref{eq:choiceNk} and $\|h\|_1\ge\alpha_{k'}$ that
\[\sum_{i=1}^{k'-1} t(H^\bullet \oplus K_{m|\ell, n}^\bullet, W_i) \geq
  \|h\|_1^{\numv{H}+m+n+\ell-1}\frac{(1-\delta_{k'-1})^{\nume{H}+mn+\ell}}{(k'-1)^{\nume{H}+mn+\ell-1}} \geq
  \left(\frac{1}{k'}\right)^{\nume{H}+mn+\ell-1},\]
which implies that \eqref{eq:mainind} holds.
In the rest,
we will assume that $\alpha_{\delta_{k'-1}/2}(W_i)<\alpha_{k'}$ for every $i\in[k']$.
 
Next suppose that there exists $i\in[k']$ such that the density $p_i$ of $W_i$ is less than $1/2k$;
by symmetry, we may assume that $i=k'$.
It follows that the density of one of the graphons $W_1,\ldots,W_{k'-1}$, say the graphon $W_1$,
is at least
\[\frac{1-1/2k-\delta_{k'}}{k'-1}
 \ge\frac{1-1/2k'-1/4k'}{k'-1}
 =\frac{4k'-3}{4(k'-1)k'}
 \ge\frac{1+1/4k}{k'}\,.\]
We obtain using Corollary~\ref{cor:universal} (if $\cut{W_1-p_1}\le\varepsilon_0$) or
Theorem~\ref{thm:bip} (if $\cut{W_1-p_1}>\varepsilon_0$) that
\[t(H^\bullet \oplus K_{m|\ell, n}^\bullet,W_1)\geq
  \left(\frac{1+1/4k}{k'}\right)^{\nume{H}+mn+\ell}\geq
  k\left(\frac{1}{k'}\right)^{\nume{H}+mn+\ell}\geq
  k'\left(\frac{1}{k'}\right)^{\nume{H}+mn+\ell},\]
which implies that \eqref{eq:mainind} (the middle inequality above
follows from \eqref{eq:choiceNkexp}).

It remains to consider the case that $p_i\ge 1/2k$ and $\alpha_{\delta_{k'-1}/2}(W_i)\le\alpha_{k'}$ for every $i\in [k']$.
In this case,
we apply to each $W_i$, $i\in [k']$,
either Corollary~\ref{cor:universal} (if $\cut{W_i-p_i}\le\varepsilon_0$) or
Theorem~\ref{thm:bip} (if $\cut{W_i-p_i}>\varepsilon_0$) to get that
\[t(H^\bullet \oplus K_{m|\ell, n}^\bullet,W_i)\geq p_i^{\nume{H}+mn+\ell}.\]
Hence, we conclude that
\[\sum_{i=1}^{k'}t(H^\bullet \oplus K_{m|\ell, n}^\bullet,W_i)\geq
  \sum_{i=1}^{k'}p_i^{\nume{H}+mn+\ell}\geq
  k'\left(\frac{p_1+\cdots+p_{k'}}{k'}\right)^{\nume{H}+mn+\ell},\]
where the last inequality follows from Jensen's inequality.
It follows that \eqref{eq:mainind} holds, which completes the proof of the claim.
\end{proof}
Since the statement of Claim~\ref{cl:step} for $k'=k$ is actually stronger than the statement of the theorem,
the proof of theorem is now completed.
\end{proof}

\section{Conclusion}
\label{sec:concl}

We finish with briefly presenting another connection of our result to Sidorenko's Conjecture.
A simple example showing that a non-bipartite graph $H$ does not have the Sidorenko property
is given when the host graph $G$ is bipartite; in such case $t(H,G)=0$.
Kohayakawa, Nagle, R\"odl and Schacht~\cite{KohNRS10} conjectured that 
having a large spare part in the host graph $G$ is the only obstacle for $t(H, G) \geq t(K_2, G)^{\nume{H}}$ to hold.
Recall that a graph is \emph{$(\rho,d)$-dense}
if any subset of at least $\rho\cdot\numv{G}$ vertices of $G$ induces a subgraph with density at least $d$;
we refer to Subsection~\ref{subsec:limits} for the relation to the notions studied in this paper.
The conjecture of Kohayakawa, Nagle, R\"odl and Schacht~\cite{KohNRS10} 
says that any graph has the Sidorenko property with respect to $(o(1),d)$-dense host graphs.

\begin{conj}[Kohayakawa, Nagle, R\"odl and Schacht~\cite{KohNRS10}]
\label{conj:KNRS}
Let $H$ be a graph, and let $\eta$ and $p$ be positive reals.
There exists a positive real $\rho$ such that
$t(H, G) \geq (1-\eta)p^{\nume{H}}$
for every sufficiently large $(\rho, p)$-dense graph $G$.
\end{conj}

Sidorenko's Conjecture would imply that Conjecture~\ref{conj:KNRS} holds for all bipartite graphs.
Among non-bipartite graphs, Conjecture~\ref{conj:KNRS} is known to hold only for some specific families~\cite{KohNRS10, Rei14, Lee21, BraSW24},
e.g, complete multipartite graphs, unicyclic graphs, cycles with a single chord,
graphs obtained by gluing complete multipartite graphs in a tree-like way, and
graphs obtained by gluing along independent sets in highly symmetric ways.
Theorems~\ref{thm:universal} and \ref{thm:bip} imply that for every graph $G$ with girth at least $50$,
there exists a connected rooted graph $H^\bullet$ containing $G$ as an induced subgraph and $n_0$ such that
the graph $H^\bullet \oplus K_{m|\ell,n}^\bullet$ satisfies Conjecture~\ref{conj:KNRS}
for all even integers $m,n,\ell\ge n_0$ such that $m$ is divisible by $5$ and $n+\nume{H}\le\ell\le mn/4$.

We finish with an open problem motivated by the techniques used in the proof Theorem~\ref{thm:universal},
where it was essential that the graph $H^\bullet\oplus K_{m|\ell,n}^\bullet$ has a long path formed by vertices of degree two.
In particular,
the techniques used in this paper would not establish the existence of a $3$-connected high-chromatic common graph.
So, it is natural to ask the following.

\begin{problem}
Is it true that for every $\ell\ge 2$ and every $k\ge 3$,
there exists an $\ell$-chromatic $k$-connected common graph?
\end{problem}

Ko and Lee~\cite{KL23} answered the above question in the affirmative
by combining our construction and the book product of graphs.
They also posed the following more general problem.

\begin{problem}[{Ko and Lee~\cite[Question 3.1]{KL23}}]
Is it true that for every $\ell\ge 2$, every $k\ge 3$ and every $g\ge 4$,
there exists an $\ell$-chromatic $k$-connected common graph with girth at least $g$?
\end{problem}

\section*{Acknowledgements}

The authors would like to thank Jon Noel and Sergey Norin for insightful comments
on possible ways of constructing common and more generally $k$-common graphs and
many detailed discussions concerning common graphs and graph limits.
The authors are also indebted to the anonymous referee for many comments that
have helped to improve the presentation of the results and proofs in the paper and
led to making the presented proofs more accessible.

\bibliographystyle{bibstyle}
\bibliography{hcommon}

\begin{thebibliography}{10}
\providecommand{\url}[1]{\texttt{#1}}
\providecommand{\urlprefix}{URL }
\providecommand{\eprint}[2][]{\url{#2}}

\bibitem{AloS16}
N.~Alon and J.~H. Spencer: The probabilistic method, Wiley Series in Discrete
  Mathematics and Optimization, 2016, fourth edition.

\bibitem{BlaR65}
G.~R. Blakley and P.~Roy: \emph{A {H}\"{o}lder type inequality for symmetric
  matrices with nonnegative entries}, Proc. Amer. Math. Soc. \textbf{16}
  (1965), 1244--1245.

\bibitem{BraSW24}
D.~Brada\v{c}, B.~Sudakov and Y.~Wigderson: \emph{Counting subgraphs in locally
  dense graphs}, preprint arXiv:2406.12418.

\bibitem{BurR80}
S.~A. Burr and V.~Rosta: \emph{On the {R}amsey multiplicities of
  graphs---problems and recent results}, J. Graph Theory \textbf{4} (1980),
  347--361.

\bibitem{CamGMS23}
M.~Campos, S.~Griffiths, R.~Morris and J.~Sahasrabudhe: \emph{An exponential
  improvement for diagonal ramsey}, preprint arXiv:2303.09521.

\bibitem{Cla92}
L.~Clark: \emph{The minimum number of subgraphs in a graph and its complement},
  J. Graph Theory \textbf{16} (1992), 451--458.

\bibitem{Con09}
D.~Conlon: \emph{A new upper bound for diagonal ramsey numbers}, Ann. of Math.
  (2) \textbf{170} (2009), 941--960.

\bibitem{ConFS10}
D.~Conlon, J.~Fox and B.~Sudakov: \emph{An approximate version of {S}idorenko's
  conjecture}, Geom. Funct. Anal. \textbf{20} (2010), 1354--1366.

\bibitem{ConFS15}
D.~Conlon, J.~Fox and B.~Sudakov: \emph{Recent developments in graph {R}amsey
  theory}, in: Surveys in combinatorics 2015, \emph{London Math. Soc. Lecture
  Note Ser.}, volume 424 (2015), 49--118.

\bibitem{ConKLL18}
D.~Conlon, J.~H. Kim, C.~Lee and J.~Lee: \emph{Some advances on {S}idorenko's
  conjecture}, J. Lond. Math. Soc. \textbf{98} (2018), 593--608.

\bibitem{ConL17}
D.~Conlon and J.~Lee: \emph{Finite reflection groups and graph norms}, Adv.
  Math. \textbf{315} (2017), 130--165.

\bibitem{ConL21}
D.~Conlon and J.~Lee: \emph{Sidorenko's conjecture for blow-ups}, Discrete
  Anal.  (2021), paper no. 2, 13pp.

\bibitem{CooKM18}
J.~W. Cooper, D.~Kr\'{a}\v{l} and T.~L. Martins: \emph{Finitely forcible graph
  limits are universal}, Adv. Math. \textbf{340} (2018), 819--854.

\bibitem{CsoHL23}
E.~{Cs{\'o}ka}, T.~{Hubai} and L.~{Lov{\'a}sz}: \emph{Locally common graphs},
  J. Graph Theory \textbf{102} (2023), 472--483.

\bibitem{Erd62}
P.~Erd\H{o}s: \emph{On the number of complete subgraphs contained in certain
  graphs}, Magyar Tud. Akad. Mat. Kutat\'o Int. K\"ozl. \textbf{7} (1962),
  459--464.

\bibitem{ErdS83}
P.~Erd\H{o}s and M.~Simonovits: \emph{Supersaturated graphs and hypergraphs},
  Combinatorica \textbf{3} (1983), 181--192.

\bibitem{Erd59}
P.~Erd{\H{o}}s: \emph{Graph theory and probability}, Canadian J. Math.
  \textbf{11} (1959), 34--38.

\bibitem{Fox08}
J.~Fox: \emph{There exist graphs with super-exponential {R}amsey multiplicity
  constant}, J. Graph Theory \textbf{57} (2008), 89--98.

\bibitem{FoxW}
J.~Fox and F.~Wei: \emph{On the local approach to {S}idorenko's conjecture}, in
  preparation.

\bibitem{FoxW17}
J.~Fox and F.~Wei: \emph{On the local approach to {S}idorenko's conjecture},
  Electron. Notes Discrete Math. \textbf{61} (2017), 459--465.
\newblock The European Conference on Combinatorics, Graph Theory and
  Applications (EUROCOMB'17).

\bibitem{Goo59}
A.~W. Goodman: \emph{On sets of acquaintances and strangers at any party},
  Amer. Math. Monthly \textbf{66} (1959), 778--783.

\bibitem{GraRS13}
R.~L. Graham, B.~L. Rothschild and J.~H. Spencer: Ramsey theory, Wiley Series
  in Discrete Mathematics and Optimization, 2013.

\bibitem{GrzLLV22}
A.~Grzesik, J.~Lee, B.~Lidický and J.~Volec: \emph{On tripartite common
  graphs}, Combin. Probab. Comput.  (2022), 907--923.

\bibitem{HanKKV23}
R.~Hancock, D.~Král', M.~Krnc and J.~Volec: \emph{Towards characterizing
  locally common graphs}, Random Structures Algorithms \textbf{62} (2023),
  181--218.

\bibitem{Hat10}
H.~Hatami: \emph{Graph norms and {S}idorenko's conjecture}, Israel J. Math.
  \textbf{175} (2010), 125--150.

\bibitem{HatHKNR12}
H.~Hatami, J.~Hladk\'y, D.~Kr\'a\v{l}, S.~Norine and A.~Razborov:
  \emph{Non-three-colourable common graphs exist}, Combin. Probab. Comput.
  \textbf{21} (2012), 734--742.

\bibitem{Hin74}
N.~Hindman: \emph{Finite sums from sequences within cells of a partition of
  {$N$}}, J. Combinatorial Theory Ser. A \textbf{17} (1974), 1--11.

\bibitem{JagST96}
C.~Jagger, P.~\v{S}\v{t}ov\'\i \v{c}ek and A.~Thomason: \emph{Multiplicities of
  subgraphs}, Combinatorica \textbf{16} (1996), 123--141.

\bibitem{KL23}
S.~Ko and J.~Lee: \emph{Common graphs with arbitrary connectivity and chromatic
  number}, J. Combin. Theory Ser. B \textbf{162} (2023), 223--230.

\bibitem{KohNRS10}
Y.~Kohayakawa, B.~Nagle, V.~R\"{o}dl and M.~Schacht: \emph{Weak hypergraph
  regularity and linear hypergraphs}, J. Combin. Theory Ser. B \textbf{100}
  (2010), 151--160.

\bibitem{KraNNVW22}
D.~Král', J.~A. Noel, S.~Norin, J.~Volec and F.~Wei: \emph{Non-bipartite
  k-common graphs}, Combinatorica \textbf{42} (2022), 87--114.

\bibitem{Lee21}
J.~Lee: \emph{On some graph densities in locally dense graphs}, Random
  Structures Algorithms \textbf{58} (2021), 322--344.

\bibitem{Lov11}
L.~Lov\'{a}sz: \emph{Subgraph densities in signed graphons and the local
  {S}imonovits-{S}idorenko conjecture}, Electron. J. Combin. \textbf{18}
  (2011), Paper 127, 21.

\bibitem{Lov12}
L.~Lov\'{a}sz: Large networks and graph limits, \emph{AMS Colloquium
  Publications}, volume~60, 2012.

\bibitem{Ram30}
F.~P. Ramsey: \emph{On a problem of formal logic}, Proceedings of the London
  Mathematical Society \textbf{2} (1930), 264--286.

\bibitem{Raz07}
A.~A. Razborov: \emph{Flag algebras}, J. Symbolic Logic \textbf{72} (2007),
  1239--1282.

\bibitem{Rei14}
C.~Reiher: \emph{Counting odd cycles in locally dense graphs}, J. Combin.
  Theory Ser. B \textbf{105} (2014), 1--5.

\bibitem{Sah23}
A.~Sah: \emph{Diagonal ramsey via effective quasirandomness}, Duke Math. J.
  \textbf{172} (2023), 545--567.

\bibitem{Sid93}
A.~Sidorenko: \emph{A correlation inequality for bipartite graphs}, Graphs
  Combin. \textbf{9} (1993), 201--204.

\bibitem{Sid96}
A.~Sidorenko: \emph{Randomness friendly graphs}, Random Structures Algorithms
  \textbf{8} (1996), 229--241.

\bibitem{Sid86}
A.~F. Sidorenko: \emph{Extremal problems in graph theory and
  functional-analytic inequalities}, Proceedings of the {A}ll-{U}nion seminar
  on discrete mathematics and its applications ({R}ussian) (1986), 99--105.

\bibitem{Sid89}
A.~F. Sidorenko: \emph{Cycles in graphs and functional inequalities}, Mat.
  Zametki \textbf{46} (1989), 72--79, 104.

\bibitem{Sid91}
A.~F. Sidorenko: \emph{Inequalities for functionals generated by bipartite
  graphs}, Diskret. Mat. \textbf{3} (1991), 50--65.

\bibitem{Sze15}
B.~Szegedy: \emph{An information theoretic approach to {S}idorenko's
  conjecture}, preprint arXiv:1406.6738.

\bibitem{Sze75}
E.~Szemer\'{e}di: \emph{On sets of integers containing no {$k$} elements in
  arithmetic progression}, Acta Arith. \textbf{27} (1975), 199--245.
\newblock Collection of articles in memory of Juri\u{\i} Vladimirovi\v{c}
  Linnik.

\bibitem{Tho89}
A.~Thomason: \emph{A disproof of a conjecture of {E}rd{\H{o}}s in {R}amsey
  theory}, J. London Math. Soc. (2) \textbf{39} (1989), 246--255.

\end{thebibliography}

\newpage

\section*{Appendix}

We provide a self-contained proof of Lemma~\ref{lm:omegaalpha} for completeness.

\begin{proof}[Proof of Lemma~\ref{lm:omegaalpha}]
Fix $\delta>0$.
We will show by induction on $r$ that
there exist positive reals $\omega_r$ and $\alpha_r$ such that
$t(K_r,W)\geq\omega_r$ or $\alpha_{\delta}(W)\ge\alpha_r$.
The statement of the lemma will follow by setting $\omega_0=\omega_{\numv{H}}$ and $\alpha_0=\alpha_{\numv{H}}$ (note that
$t(K_r,W)\leq t(H,W)$ for every $r$-vertex graph $H$).

The base of the induction is the case $r=1$, which holds with $\omega_1=1$ and any $\alpha_1 \in (0,1]$.
We next present the induction step when $r\ge 2$.
Assuming the existence of $\omega_{r-1}$ and $\alpha_{r-1}$,
we will show that it is possible to set $\omega_r=(1-\delta)\delta^{r-1}\omega_{r-1}$ and $\alpha_r=\delta\alpha_{r-1}$.

Consider a graphon $W$ and
let $A$ be the set of $x\in [0,1]$ such that $\deg_W(x)\le\delta^2$.
If $\mu(A)\ge\delta$,
then the density of the graphon $W[A]$ is at most $\mu(A)\delta^2/\mu(A)^2\le\delta$,
which yields that $\alpha_{\delta}(W)\ge\delta\ge\alpha_r$.
Hence, we can assume that $\mu(A)\le\delta$ in what follows.

Let $f_x:[0,1]\to [0,1]$ for $x\in [0,1]$ be the function defined as $f_x(y)=W(x,y)$.
Observe that
\[t(K_r,W)=\int_{[0,1]}\left(\int_{[0,1]}f_x(y)\dd y\right)^{r-1}t(K_{r-1},W[f_x])\dd x.\]
It follows that if $t(K_{r-1},W[f_x])\geq\omega_{r-1}$ for every $x\in [0,1]\setminus A$,
then $t(K_r,W)\geq (1-\mu(A))\delta^{r-1}\omega_{r-1}\geq\omega_r$.
Otherwise, there exists $x\in [0,1]\setminus A$ such that $\alpha_{\delta}(W[f_x])\geq\alpha_{r-1}$,
which means that there exist $g:[0,1]\to [0,1]$ such that
\[\frac{\int_{[0,1]}f_x(y)g(y)\dd y}{\int_{[0,1]}f_x(y)\dd y}\ge\alpha_{r-1}\qquad\mbox{and}\qquad
  \frac{\int_{[0,1]^2}f_x(y)f_x(z)g(y)g(z)W(y,y)\dd y\dd z}{\left(\int_{[0,1]}f_x(y)g(y)\dd y\right)^2}\le\delta.\]
It follows that the function $h:[0,1]\to [0,1]$ defined as $h(y)=f_x(y)g(y)$ satisfies that
$\|h\|_1\ge\alpha_{r-1}\delta=\alpha_r$ and the density of $W[h]$ is at most $\delta$.
Hence, $\alpha_{\delta}(W)\ge\alpha_r$ as required.
This completes the proof of the induction step and so the proof of the lemma.
\end{proof}

\end{document}